\documentclass[11pt]{amsart}

\usepackage{amsmath,amsthm, amssymb, color}
\usepackage{graphicx}
\usepackage{fullpage}
\usepackage[normalem]{ulem}

\usepackage{amsmath, amsthm, amssymb, bm, mathtools}
\usepackage{enumerate}

\newcommand{\R}{{\mathbb {R}}}

\newcommand{\Z}{{\mathbb Z}}

\newtheorem{Theorem}{Theorem}

\newtheorem {lemma} [Theorem]    {Lemma}
\newtheorem {corollary}  [Theorem]    {Corollary}

\newtheorem {proposition}[Theorem]    {Proposition}
\newtheorem {theorem}[Theorem]    {Theorem}
\newtheorem{remark}[Theorem]{Remark}

\newcommand {\eps}{\varepsilon}

\title{Intensity doubling for Brownian loop-soups in high dimensions}
\author{Titus Lupu and Wendelin Werner}
\address{CNRS/Universit\'e Paris-Sorbonne and University of Cambridge}

\begin{document}

\begin {abstract}
We derive an  intensity doubling feature of critical Brownian loop-soups on the cable-graphs of $\Z^d$ for $d \ge 7$ that can be described as follows: In the box $[-N, N]^d$ (and with a probability that goes to $1$ as $N$ goes to infinity), the set of all clusters of Brownian loops (in this Poissonian collection of such loops) that do contain proper self-avoiding cycles of diameter comparable to $N$ can be decomposed into two  identically distributed families: (a) The collection of clusters that do contain a large Brownian loop from the loop-soup (and therefore do automatically contain such a large cycle) (b) The collection of clusters that contain no macroscopic loop from the loop-soup (more specifically, no loop of diameter greater than $N^{\beta}$ when $\beta \in ( 4/ (d-2), 1)$ is fixed)  but nevertheless contain a large cycle. In particular, due to the fact that these two families are asymptotically identically distributed,  large cycles formed in case (b) by chains of small Brownian loops (i.e., all  of diameter smaller than $N^\beta$) will look like large Brownian loops themselves, and form a second independent ``ghost''  critical loop-soup in the scaling limit.
Reformulated in terms of the Gaussian free field on such cable-graphs, this shows that large cycles in the collection of its sign clusters will converge in the scaling limit to a Brownian loop-soup with twice the ``usual critical intensity''.
This result had been conjectured in \cite {L3}; our proof builds  heavily on the recently derived switching property for such critical cable-graph loop-soups from \cite {W3}.  \end {abstract}

\maketitle 

\section {Introduction}

\subsubsection*{Background on cable-graph loop-soups}
In transient cable graphs such as those of $\Z^d$ for $d \ge 3$, one can define a natural Brownian loop measure and perform a Poisson point process of such Brownian loops with intensity given by this particular measure.
This ``Brownian loop-soup'' (in the terminology of \cite {LW}) turns out to be closely related to the Gaussian Free Field (GFF) in that space (see \cite {LJ1,LJ2}) and to questions about the determinant of the Laplacian (see \cite {LJ2,L,WP,LJ3} and the references therein).

We will briefly recall the exact definition of the Brownian loop-soup in Section \ref {LSdef}, but let us already give some heuristic intuition about its discrete analog.  If one considers the discrete graph $\Z^d$, it is natural to consider the measure $\nu$ that puts a mass $(2d)^{-n(\gamma)}$ to each nearest-neighbour loop $\gamma$ with $n(\gamma)$ steps (the loop is here considered to be  ``not rooted'' i.e., with no marked starting point, and also ``not oriented'' so that a loop and its ``time-reversal'' are considered to be the same -- so it a nearest-neighbour from $\Z / n \Z$ into $\Z^d$ defined modulo circular time-shift and time-reversal).  The Brownian loop-measure $\nu_B$ that we will use here is the simply the Brownian analog of this measure on discrete unrooted unoriented loops, where one now considers Brownian loops on the cable graph (which is the union of all the one-dimensional edges joining neighboring points in $\Z^d$) instead of discrete random walk loops.

The Brownian loop-soup is then a Poisson point process of unrooted and unoriented Brownian loops with this particular intensity. In other words, random circular Brownian structures appear ``independently'' on the cable-graph, with intensity given by the Brownian loop-measure. This can appear somewhat convoluted at first sight, but from a ``continuum'' metric cable-graph perspective (where points/sites play no particular role), it is a particularly intrinsic percolation model. Indeed, Brownian paths are canonically defined from the metric and the Brownian loop measures furthermore remains the same if one multiplies the metric by any constant (which contrasts with usual Bernoulli bond-percolation on discrete graphs say, that can be viewed as coming from a Poisson point process of points/cuts on the cable graph, which is arguably the other very canonical model -- in this case, multiplying the metric by a constant amounts to changing the percolation parameter).

It turns out that if one samples such a  loop-soup on the cable-graph of $\Z^d$ for $d \ge 3$, for every given point $x$, there will almost surely exist infinitely many small loops that go through $x$ (but only finitely many loops of diameter greater than $\epsilon$, for any fixed $\epsilon$). On the other hand, it is not difficult to see there will almost surely exist many {\em exceptional} points on the cable graph that are visited by no loop of the loop-soup.

It is interesting to look also at Brownian loop-soups with intensities given by a multiple $c \not= 1$ of $\nu_B$, but the case $c=1$ that we are focusing on in the present paper has some important special features. It is worth stressing here to avoid confusion that in the literature (for instance in \cite {LJ1,L1}), one often starts with the natural measure on {\em oriented} loops, which corresponds to $\tilde \nu_B = 2 \nu_B$, so that the $c=1$ case that we are considering is the one with intensity $\alpha \tilde \nu_B$ for $\alpha =1/2$.

There are several related reasons that make it particularly natural and particularly physically relevant to then look at the clusters of Brownian loops (the connected components of the union of all the loops in the loop-soup) for $c=1$: One is the rewiring properties from \cite {W1} that indicate that the way in which clusters are subdivided/covered into individual loops is in some sense ``uniform and local'' -- this can be viewed as a ``bosonic feature'' of the model.  Another feature is the relation of the loop-soup with the Gaussian free field (GFF) and its correlations.
There has been a rather intense activity on this particular model on cable graphs in this past decade -- references include \cite {Aid1, CD0, CD00, CD1, CD2, CD3, DW, DPR0, DPR1, DPR2, DPR, GJ, GN, L1, L2, L3, LST, LW1, LW3, P, W1, W2, W3}.
Results that will be particularly useful in the present paper are the explicit formula for the probability that two given points $x$ and $y$ belong to the same cluster of Brownian loops (we denote this event by $x \leftrightarrow y$) in terms of the Green's function $g(x,y)$ in the graph (see \cite {L1}) or the switching property (see \cite {W3}) that allows to describe the law of the clusters when one conditions on the event $x \leftrightarrow y$. Both these features are specific to this particular percolation-of-loops model. One can interpret these specific features as part of the reason why this percolation model is physically natural, or as the technical reason for which one ends up being able to say more about it than for ordinary percolation.
We choose to first present our main result in the Brownian loop-soup framework  -- we will briefly come back to what it says about the GFF at the end of this introduction.

These loop-soup clusters are of interest in any dimension (in the planar case, this ends up being part of the conformal invariance SLE/CLE realm, see for instance \cite {SW}). The focus of the present paper with be the case where the dimension of space is high (i.e., $d \ge 7$ in the case of $\Z^d$). There, the proportion of large Brownian loops versus small Brownian loops in a given large box will be smaller and the clusters of Brownian loops (the connected components of the union of all Brownian loops) will (or are believed to) share most features of the clusters of critical Bernoulli percolation in that spatial dimension. For instance, the probability that two given points $x$ and $y$ are in the same cluster decays like a constant times $1/|x-y|^{d-2}$ when $|x-y| \to \infty$ (this follows directly from the explicit formula mentioned above in the case of loop-percolation) -- the same result is known to hold for nearest-neighbour Bernoulli percolation for $d \ge 11$ and sufficiently spread-out Bernoulli percolation for $d \ge 7$, see \cite {HvdH} and the references therein. Some recent references dealing with this high-dimensional regime of Brownian loop-percolation include \cite {W2,CD1,DPR,GJ,CD2} (see more generally the recent paper \cite {CD2} for a more complete reference list).

\subsubsection*{Loop-soups in high dimensions}
The main goal of the present paper is to derive a novel and somewhat surprising feature of this model in such high dimensions that also provides new insight into critical models in high-dimensions in general. Before describing it, let us recall some further simple observations from \cite {W2}.
Consider a critical Brownian loop-soup in $\Lambda_N := [-N,N]^d$ (viewed as a cable graph with the edges of $\Z^d$ in that box) for $d \ge 7$. More precisely, this is the collection all the Brownian loops in a loop-soup in $\Z^d$ that stay in $\Lambda_N$ (it can also be defined directly in $\Lambda_N$ where the boundary of the box plays the role of killing points for the Brownian motion).
Then:
\begin {itemize}
 \item On the one hand, the definition of the loop-soup shows immediately that for any $\eps$, the number of Brownian loops of diameter greater than $\eps N$ is tight [it converges in law to a Poisson random variable as $N \to \infty$]. So, there will for instance typically exist only a handful of Brownian loops of diameter greater than $N/2$.
 \item On the other hand, if one looks at the set of large clusters, the situation will be very different. Indeed, just as in the case of ordinary critical Bernoulli percolation in high dimensions (see \cite {Ai}), it is possible (and in that case not difficult, see \cite {W2}) to show that ``large clusters will proliferate'' (in the terminology used by Michael Aizenman in \cite {Ai}) meaning that for each small $\eps$, the number $n_N=n_N (\eps)$ of clusters of diameter greater than $\eps N$ in $\Lambda_N$ will explode like $N^{d-6 + o(1)}$.
 \item
 It therefore follows (see \cite {W2}) that most of these clusters will contain no Brownian loop of diameter greater than $N^\alpha$ at all, as soon as $\alpha > 6/d $. This is simply due to the fact that there are of order $N^{d(1- \alpha)}$ such loops, while the number of macroscopic clusters (which is $N^{d-6 +o(1)}$) is much larger. As we shall see in the paper, it is actually also very easy to see that (typically, when $N \to \infty$) no macroscopic cluster will contain more than one macroscopic Brownian loop.
\end {itemize}
In a nutshell, this simple argument indicates that a ``typical'' large cluster will be composed only of Brownian loops of diameter smaller than $N^\alpha$ [as we shall see, it is in fact not difficult to show that this will actually be the case for all $\alpha > 4 / (d-2)$]. This type of ideas/results has been used and developed in subsequent recent work, such as \cite {CD1}.

\subsubsection*{Results of this paper}
We will be interested in the special {\em exceptional} macroscopic clusters that happen to be topologically non-trivial at macroscopic scale, i.e., that contain ``proper'' macroscopic cycles (i.e. large self-avoiding cycles that are not close to be tree-like at macroscopic scale, in a sense that we will make precise) -- to avoid confusion, we will try to use the word {\em loops} for Brownian loops in a loop-soup and the word {\em cycles} for subsets of clusters (that are not necessarily Brownian loops in the loop-soup). A macroscopic Brownian loop will typically itself contain such a macroscopic ``proper cycle'', so that (combined with the fact that no cluster will contain more than one macroscopic loop) we can infer that the number  macroscopically topologically non-trivial clusters will be at least as big as the number of macroscopic Brownian loops. The question is to describe (if they exist!) the number and shape of those large cycle-containing-clusters that do {\em not} contain any large Brownian loop from the loop-soup (so that the cycle would have to be created by a chains of smaller loops).

In a nutshell, the answer will be that such clusters will exist, and that they will be {\em as numerous and have a similar shape as the set of clusters that do contain a large Brownian loop.}
In somewhat less loose terms (a precise formal statement will be given in the next section): For any $d \ge 7$, if $\beta \in (  4 / (d-2), 1)$  is fixed, then, for each $\eps >0$,
in the limit when $N \to \infty$, one can decompose the set of clusters in $[-N,N]^d$ that do contain  ``proper'' (i.e., not close to tree-like) macroscopic cycles with diameter greater $\eps N$ into two (almost) identically distributed families:
\begin {enumerate}
 \item The set of clusters that contain exactly one macroscopic Brownian loop of the loop-soup (and this Brownian loop will already contain a cycle, which will be essentially the only cycle present in the cluster) as discussed above -- this case is depicted on the left of Figure \ref {fig:1},
 \item the set of clusters that do contain no Brownian loop of diameter greater than $N^\beta$ (as depicted on the right of Figure \ref {fig:1}),
\end {enumerate}
and there will be no other ones (with a probability that goes to $1$ as $N \to \infty$).

One main point in this statement is that the set of clusters satisfying (1) and (2) are (almost) identically distributed. So, the number of clusters of type (2) will converge to a Poisson random variable and chains of loops of size smaller than $N^\beta$ will form a collection of cycles that will look like a Poisson cloud of macroscopic Brownian loops.

In fact, for each individual cluster,
the choice of being of the first or second kind is made by (almost) fair independent coins (one for each cluster -- we will explain that the independence follows readily from the Poissonian feature of the loop-soup); a more precise way to formulate this $N \to \infty$ statement is that, with a probability that tends to $1$ as $N \to \infty$, when one conditions on all these macroscopic clusters, the conditional probabilities for each cluster that contains a macroscopic cycle to be of type (1) and (2) respectively will both be  $1/2 + o(1)$, and the conditional probability of being in neither case will be $o(1)$.

In fact, for Case (1), the clusters will also contain no Brownian loop of diameter greater than $N^\beta$ other than the one of diameter larger than $\eps N$ that it already contains by definition of this case.
In Case (2), the large cycles will in fact be created by Brownian loops of diameter much smaller than $N^{\beta}$ (i.e., with probability going to $1$, they will be created by loops smaller than $N^\gamma$  for any given  $\gamma \in (2/ (d-4) , \beta)$ (i.e., loops of diameter larger than $N^\gamma$ but smaller than $N^\beta$ can be in the cluster, but will not contribute to the creation of the large cycle).

\begin{figure}[h]
  \centering
  \includegraphics[width=1\textwidth]{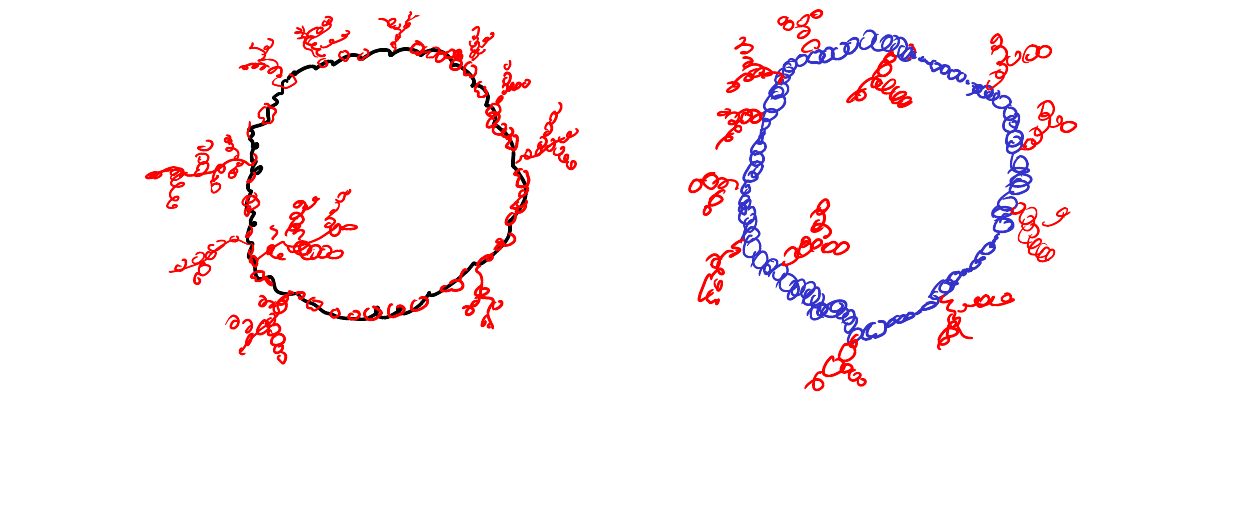}
  \caption{\label{fig:1}Each cluster topologically non-trivial cluster containing a large cycle will be decomposed in one of these two ways with (almost) equal probability: Either the cycle is due to one big Brownian loop (left picture), or the cycle is created by a chain of very small Brownian loops (right picture).}
  \end{figure}

Furthermore, we will see that one can define in some deterministic way a macroscopic cycle in each of these macroscopic clusters in such a way that the collection of these cycles will converge (in the scaling limit, i.e., when rescaled by $1/N$) to a continuum Brownian loop-soup in  $[-1, 1]^d$ with twice the intensity of the usual standard loop-soup. Half of the cycles will correspond to the actual macroscopic loops in the loop-soup that one started with, and half of the cycles will correspond to the macroscopic cycles created by long chains of small Brownian loops (of diameter smaller than $N^{\gamma}$ each).

As explained in \cite {L1}, clusters of critical Brownian loop-soups can be equivalently  described as the excursions sets away from $0$ of a Gaussian Free Field (GFF) on the cable graph. This makes it possible to formulate a consequence of our main result in terms of the GFF  as follows:  Large proper cycles in the excursions sets of the GFF in the cable graph of $[-N,N]^d$ will converge in the scaling limit to a Brownian loop-soup with intensity equal to twice that of the usual critical loop-soup in $[-1, 1]^d$.
This is the ``intensity doubling conjecture'' formulated in \cite {L3} (based on some formulas for ``twisted'' loop-soups and Gaussian free fields) -- that is therefore proved in the present paper for all $d \ge 7$.

\medbreak

Let us say a few words about how our we will derive these results: Our proofs will be heavily based on one of the switching properties derived in \cite {W3} for such loop-soups applied to loop-soups in well-chosen subgraphs of $[-N,N]^d$. They provide an illustration of how these switching properties can be used to obtain rather directly results about the geometry of loop-soup clusters that are either out of reach or technically very challenging for other percolation models. Here is a flavour of some of the main ideas here: When a loop-soup cluster $C$ in $[-N,N]^d$  contains a large cycle $\Gamma$ that wraps around some $(d-2)$-dimensional affine subspace $\Delta$ (this is the $d$-dimensional analog of a cycle wrapping around a line in three dimensions) while staying at distance $\eps N$ from $\Delta$, either $\Gamma$ intersects a large Brownian loop (which is part of the loop-soup, it is then part of that cluster) of diameter greater than $\eps N /2$, or it does not intersect any such Brownian loop, in which case the cycle $\Gamma$ will be part of a cluster of the loop-soup in the sub-cable-graph of $[-N,N]^d$  obtained by removing all points that lie at distance smaller than $\eps N /2$ from $\Delta$. The cycle-version of the switching property from \cite {W3} applied to this restricted loop-soup  then implies that in the latter case, the conditional probability that there exists a Brownian loop that wraps around $\Delta$ while staying at distance at least $\eps N /2$ from it is at least $1/2$. So, we see that (in both options) the conditional probability that $C$ contains a Brownian loop of diameter greater than $\eps N /2$ is at least $1/2$. A more refined analysis of the shape and structures of the clusters (this will involve looking at loop-soups in the universal cover of domains like $\Z^d \setminus \Delta$) that do contain large Brownian loops will then (roughly) indicate that the only proper cycles that they contain will be very close to these Brownian loops themselves, which will allow to deduce the above dichotomy from the switching property.

\medbreak

Let us finally comment about what this result says for critical models in high-dimensions in general: For many critical models from statistical physics, there is usually an ``upper-critical dimension'' of space above which the model will behave in a ``mean-field way''. This is often related to the inter-related facts that (1) the correlation functions behave like that of the GFF, (2) the geometric underlying macroscopic clusters tend to proliferate and there is no geometric scaling limit, (3) these clusters will typically be tree-like on large scale (with no macroscopic loops). What the present paper shows is that in fact, if one looks at the exceptional clusters that are topologically non-trivial (instead of looking at all/typical clusters), scaling limits will in fact arise in the form of Brownian loop-soups. See the paper \cite {CW} (and further forthcoming work) for the case of ordinary percolation.

\section {The setup and the more precise statement of the main result}

\label {LSdef}
As promised, we now briefly the definition of the Brownian loop-soup on transient cable-graphs: The cable-graph consists of the union of all the closed edges of the graph (viewed as one-dimensional segments), and it is then easy to define Brownian motion on this locally one-dimensional structure (when the Brownian motion is at a vertex where several edges meet, it essentially tosses a uniform coin to decide in which of the edges to move when it makes an excursion away from the vertex). From this, one can then define for each point $x$ in the cable-graph, a Brownian excursion measure $\mu_x$ away from $x$ in the cable graph, which is a measure on parameterized oriented loops $\gamma$ from $x$ to $x$ (we call this a loop that is rooted at $x$)-- a loop $\gamma$ has a positive and finite time-duration $T(\gamma)$. The total mass of $\mu_x$ is then the Green's function $G(x,x)$ in the cable graph. In the set-up that we are looking at in the present paper, it is natural to consider the loops to be ``non-oriented'' meaning that a loop and the loop traced backwards will be considered to be the same. This makes it natural to consider the measure $\mu_x / 2$ instead of $\mu_x$.  The natural loop-measure is  then defined by
$$ \mu (A) := \int_{\mathcal G} dx \int \frac{d\mu^x (\gamma)}{2} \frac{ 1_{\gamma \in A}}{T(\gamma)},.$$
where we view $\mu$ as a measure on the set of unrooted and unoriented loops -- i.e., two rooted oriented loops are identified if they are obtained by monotone circular time-reparameterization of each other) and $dx$ denotes the one-dimensional Lebesgue measure on ${\mathcal G}$. The term $1/2$ can be seen as coming from the oriented/unoriented feature mentioned in the previous section, while the $1/T(\gamma)$ term corresponds to the idea that an unrooted loop of time-length $T$ has $T$ times more options for choosing its root than a loop of time-length $1$ (so that this term is here to compensate the natural overcounting due to that -- one can view the point $x$ in the previous construction as being chosen uniformly (with respect to time) on the unrooted loop). One standard feature (going back to the the original construction from \cite {LW} in the continuum) is that it is easy to choose to ``root'' a loop at some special point (for instance, the point where the first coordinate is the largest) and to then describe the law of the loop viewed from that point. One can also (and we will use this in the present paper) choose some point $x$ on the cable-graph and describe the measure $\mu$ restricted to the set of loops that go through $x$ in terms of Brownian loops ``rooted at $x$''.

So, in a nutshell, $\mu$ is a natural measure on unoriented unrooted loops in the cable-graph, and the loop-soup that we will consider in the present paper is a Poisson point process with intensity $\mu$. Heuristically speaking, this means that loops appear independently in a loop-soup, with an ``intensity'' provided by $\mu$. As mentioned in the introduction, when the cable-graph is transient, it will turn out that (a) for every given point in the cable-graph, there will almost surely be infinitely many small Brownian loops in the loop-soup that go through this point, but only finitely many ones of diameter greater than $1/2$ say, and (b) there will be exceptional points that are visited by none of the loops in the loop-soup (let us call $Z$ this set). We can then define clusters of Brownian loops to be the connected components of the unions of all loops in the loop-soup (or equivalently the connected components of the complement of $Z$).

Throughout the paper, when we talk about the loop-soup in $\Lambda_N = [-N,N]^d$, we will mean a loop-soup in the cable-graph of $\Z^d$ restricted to the set of loops that stay in $\Lambda_N$ (as noted in the final section, our results can then be directly used to obtain the similar results if one considers a loop-soup in the whole of $\Z^d$ and considers the large cycles or the large cycles it creates that are contained in $\Lambda_N$.

\medbreak

To state the result precisely, we need to define our set of ``proper non-tree-like'' cycles, which is what we now do:
One way to proceed is to fix some small $\eps >0$ and to define ${\mathcal C}(\eps, N)$ to be the set of loop-soup clusters (for the cable-graph loop-soup in $\Lambda_N:= [-N, N]^d$ for $d \ge 7$) that do contain a simple cycle $\Gamma$ such that there exists a plane $\Pi$ such that the orthogonal projection of $\Gamma$ on $\Pi$ has a non-trivial index around some point $x$ in this plane and also stays at distance at least $\eps N$ from it (so in particular, this cycle $\Gamma$ has necessarily a diameter of at least $2 \eps N$). This definition ensures that the cycle is macroscopic and also not tree-like at macroscopic level (as it circumnavigates around an entire ``tube'').

We will similarly define the collection ${\mathcal B}(\eps, N)$ of single loops that do contain a simple cycle with the same conditions. So, if a cluster contains a Brownian loop that belongs to ${\mathcal B}(\eps, N)$, this cluster will automatically belong to ${\mathcal C}(\eps, N)$.

Recall that when $d \ge 4$ (and therefore when $d \ge 7$), Brownian motion in $\R^d$ is almost surely a simple curve (i.e., with no double points) and that a large random walk loop in $\Z^d$ will in fact typically contain a self-avoiding cycle of diameter comparable to that the loop. Furthermore, since the projection of the loop on the plane $\R^2 \times \{ 0\}^{d-2}$ is a two-dimensional loop, one will be able to find a point $x$ in this plane, such that this projection has a non-trivial index around it and stays at distance comparable to $N$ from it (this corresponds to properties of two-dimensional Brownian motion). In other words, all Brownian loops of diameter at least $\eps' N$ in the loop-soup will end up being in ${\mathcal B} (\eps, N)$ for some $\eps$ comparable to $\eps'$ (i.e., the probability will converge to $1$ when $\eps'$ is fixed and $\eps \to 0$, uniformly with respect to $N$).

We will condition on the clusters in ${\mathcal C}(\eps, N)$ (viewed as sets, not as collection of individual loops) and discuss aspects of the conditional distribution of the Brownian loops in each of these clusters.

An important first observation  is that the ways in which the different loop-soup clusters are decomposed/covered by Brownian loops are conditionally independent (i.e., the conditional law of the decomposition of $C_i$ given all of the $C_j$'s is a function of $C_i$ only).
One way to justify this goes as follows: We know that any two different loop-soup clusters on a cable-graph are always at positive distance from each other (this is just due to the relation with the GFF on the cable graph -- excursions of the GFF can be made to coincide with the collection of clusters --, combined with the fact that the GFF on the cable graph has no isolated zeros -- indeed its restriction to each edge is distributed like a one-dimensional Brownian bridge). It therefore follows readily from the Poissonian construction of the Brownian loop-soup that if one observes a finite set of (different) loop-soup clusters $C_1, C_2, \ldots, C_n$, the $n$ collections $(L_1, \ldots, L_n)$ of Brownian loops that constitute $C_1, \ldots, C_n$ are conditionally independent [here $L_j$ is the collection of all the Brownian loops in the cluster $C_j$].
More precisely, if we fix $n$ points $x_1, \ldots, x_n$ on the cable graph and define $C_1, \ldots, C_n$ to be the clusters that contain these points, then on the event that $C_1 \subset I_1, \ldots, C_n \subset I_n$ for disjoint closed sets $I_1, \ldots, I_n$, each $C_i$ and its decomposition into Brownian loops can be read off from the collection of loops that stay in $I_i$. Since these $n$ collections of loops are independent, the collections $(L_1, \ldots, L_n)$ are indeed conditionally independent on this event.

We can now properly formulate the main statement of this paper as a comparison/relation between the set ${\mathcal C}(\eps, N)$ and the subset of  ${\mathcal C}(\eps, N)$ consisting of those clusters that happen to contain a large Brownian loop in ${\mathcal B}(\eps , N)$ (that already ensures by itself that its cluster is in ${\mathcal C}(\eps, N)$).

\begin {theorem}[Intensity doubling for $d \ge 7$]
\label{thm}
For each $\eps > 0$, the number of clusters in ${\mathcal C}(\eps, N)$ remains tight when $N \to \infty$ and does converge in law to a Poisson random variable. Furthermore, for each $\beta
> 4 / (d-2)$ and $\gamma > 2 / (d-4)$,  one can find an explicit deterministic $u_N \to 0$ such that the following happens with a probability that tends to $1$ as $N \to \infty$: For each cluster $C \in {\mathcal C} (\eps, N)$:
\begin {enumerate}
 \item The conditional probability (given the cluster) of the following event is in $[1/2 - u_N, 1/2 + u_N]$: $C$ contains exactly one Brownian loop in ${\mathcal B}(\eps, N)$ and no other Brownian loop of diameter greater than $N^\beta$.
 \item The conditional probability (given the cluster) of the following event is in $[1/2 - u_N, 1/2 + u_N]$: $C$ contains no Brownian loop of diameter greater than $N^\beta$, and it contains a large cyclic chain of Brownian loops that are all of diameter smaller than $N^\gamma$ that ensures that $C$ is in ${\mathcal C}(\eps, N)$.
 \item The conditional probability (given the cluster) that neither (1) nor (2) holds is therefore smaller than $2u_N$.
\end {enumerate}
Furthermore, for each cluster $C$ in ${\mathcal C}(\eps, N)$, one can define deterministically a cycle $\Gamma(C)$ in $C$ in such a way that the probability that
 there exists $C$ in ${\mathcal C}(\eps, N)$ for which the cluster contains a Brownian loop in ${\mathcal B}(\eps, N)$ that is at Hausdorff distance greater than $N^\beta$ of $\Gamma (C)$ is smaller than $u_N$ [so in other words, if one is in case (1), then the large Brownian loop has to be close in Hausdorff distance (compared to the scale $N$) to the given cycle $\Gamma (C)$ with a very high probability).
\end {theorem}

The closer $\beta$ and $\gamma$ are  to the values $4/(d-2)$ and $2/(d-4)$ (respectively), the stronger the statements (1)-(2)-(3) are.
The statement is not void since one can always find $\beta$ and $\gamma$ smaller than $1$ that satisfy the conditions in the theorem when $ d \ge 7$, while (and this is not a surprise here) when one formally plugs in $d=6$ into the formula, the bounds $4/(d-2)$ and $2/(d-4)$ both  take the critical value $1$. The range of admissible values for $\beta$ and $\gamma$ increases with $d \ge 7$. In particular, any values $\beta > 4/5$ and $\gamma > 2/3$ will work for all $d \ge 7$. As we will explain at the end of the paper, the values of these two thresholds $4/(d-2)$ and $2 / (d-4)$ are not surprising and can be heuristically explained rather simply.

\medbreak

We conclude this section with the following remark related to the link with the GFF:
\begin {remark}
 We could also (as in \cite {W3}) additionally condition on the total occupation times of the Brownian loops on the clusters on top of the clusters themselves, and the very same result then still holds (i.e., in the limit, for each of the macroscopic cycle-containing clusters, the conditional probability that it contains a macroscopic loop will tend to $1/2$ independently of the local time profile) -- this will simply be due to the fact that the switching property that we will use is also valid when one does additionally condition on the local time profile.
Recall that this local time profile is very natural in view of the relation to the Gaussian Free Field on the cable graph (the profile is the square of the GFF and the clusters will be the excursion sets away from $0$ of the GFF).
\end {remark}

\section {Preliminary estimates}

Let us first collect some rather simple preliminary facts about large Brownian loops within a loop-soup and loop-soup clusters in $\Lambda_N:=[-N,N]^d$ for $d \ge 7$. In this section, we will not yet focus on cycle-containing clusters.

\subsection {A first remark/warm-up}

We will repeatedly and implicitly use the following type of simple estimates in our discussions and proofs: Suppose that one considers a loop-soup and wants to estimate the probability that a given integer point (i.e., a point with integer coordinates) is part of a Brownian loop of diameter at least $N^a$ for some $a <1$. This probability is bounded by the expected number of loops of that diameter that it is part of, and this expected value is just the total mass (for the Brownian loop measure) of the set of loops that pass through $x$. This total mass can then simply be expressed by choosing to root these loops at $x$ (see for instance \cite {WP}), and can then be bounded by a constant times the probability that a random walk started from $x$ reaches distance $N^a /4$ from $x$ and then comes back to $x$ (before escaping to infinity). One therefore readily gets an upper bound of a constant times $1/N^{a(d-2)}$.
Summing over all integer points in $\Lambda_N$ gives an upper bound of a constant times $N^{d(1-a)} \times N^{2a}$. This corresponds to the intuition that the main contribution to the number of points that belong to Brownian loops of diameter greater than $N^a$ will come from loops of diameter of order $N^a$ -- there are $N^{d(1-a)}$ of them and each will have a mean number of points of order $N^{2a}$.

Similarly, we will be looking at quantities like the expected sum over all loops of diameter at least $N^a$ in $\Lambda_N$ of the number of pairs of integer points on this loop. The previous intuition indicates that this quantity will be bounded by a constant times $N^{d(1-a)} N^{4a}$. This can indeed be easily proven in a similar way as before: For each $x$ and $y$, we can bound the mass of the set of loops of diameter greater than $N^a$ that go through $x$ and $y$ by comparing it with the probability that a random walk starts from $x$, then reaches distance $N^a/4$ from $x$ (if $y$ is at distance smaller than $N^a /8$ from $x$, otherwise we can leave this one step out), then visits $y$ and then comes back to $x$, which gives a bound of a constant times $(1/N^{a(d-2}) \times (1 / | x-y|^{d-2})$, that indeed provides the right bound when summing over all $x$ and $y$'s.

\subsection {The shape of large Brownian loops}

The following two types of facts will be helpful in our proof:

\medbreak

The projection of a long  cable graph Brownian loop (i.e. time of order $N^2$ or diameter of order $N$) on a two-dimensional coordinate plane (say the first two coordinates) will look like a long two-dimensional loop on the square lattice. It will therefore have a non-zero index around many points and have both non-zero index and stay at distance comparable to $N$ i.e. greater than any $o(N)$ of many points (with high probability as $N \to \infty$). There are many ways to formulate this precisely. One fact that we will use at some point is that for all fixed $a<1$,  with high probability as $N \to \infty$, all Brownian loops of diameter greater than $\eps N$ in the loop soup have the property that their projection on the first two coordinates has a non-zero index around some point (the point can depend on the loop) while staying at distance at least $4N^a$ from it.

\medbreak
The second type will be related to pinching points: A long Brownian loop on the cable graph of diameter comparable to $N$ will not be ``close to pinching'', meaning that the probability that one can find four points $x$, $y$, $x'$, $y'$ on the loop that are ordered circularly with respect to the time-parametrization of the  loop  with $d(y,y') \le N^{c}$ while $d(y,x) \ge N^{b}$ and $d(y,x') \ge N^{b}$ is going to $0$ (provided $b$  has been chosen close enough to $1$ once $c<1$ is given).
This statement can  be strengthened into statements about ``all Brownian loops of size greater than $N^{a}$ in the loop-soup in $\Lambda_N$ are not close to pinching'' in the following sense, that will be instrumental in our proofs:

\begin {lemma}[No $(a,b,c)$ pinching, see Figure \ref {fig:2}]
\label {pinch}\label {L1}
 Suppose that $a,b,c$ with $c < b \le a$ are in the interval $(0,1)$ and are such that $ \nu:= (a+b)(d-2) -( d + c (d-4)) > 0$ [we will refer to this condition as the (a,b,c) condition]. Let  ${\mathcal E}_1 (a,b,c,N)$ denote the event that there exists a Brownian loop in the loop-soup in $\Lambda_N$ with the property that there are points $x=B(i)$ and $y=B(j)$ on the loop that lie at distance smaller than $N^c$ from each other and such that the two portions of the loop between these two points (so the two arcs joining $x$ and $y$) respectively have a diameter greater than $N^a$ and $N^b$. Then, $P[ {\mathcal E}_1 (a,b,c, N)]  \le O( N^{-\nu})$ as $N \to \infty$.
\end {lemma}

\begin{figure}[h]
  \centering
  \includegraphics[width=\textwidth]{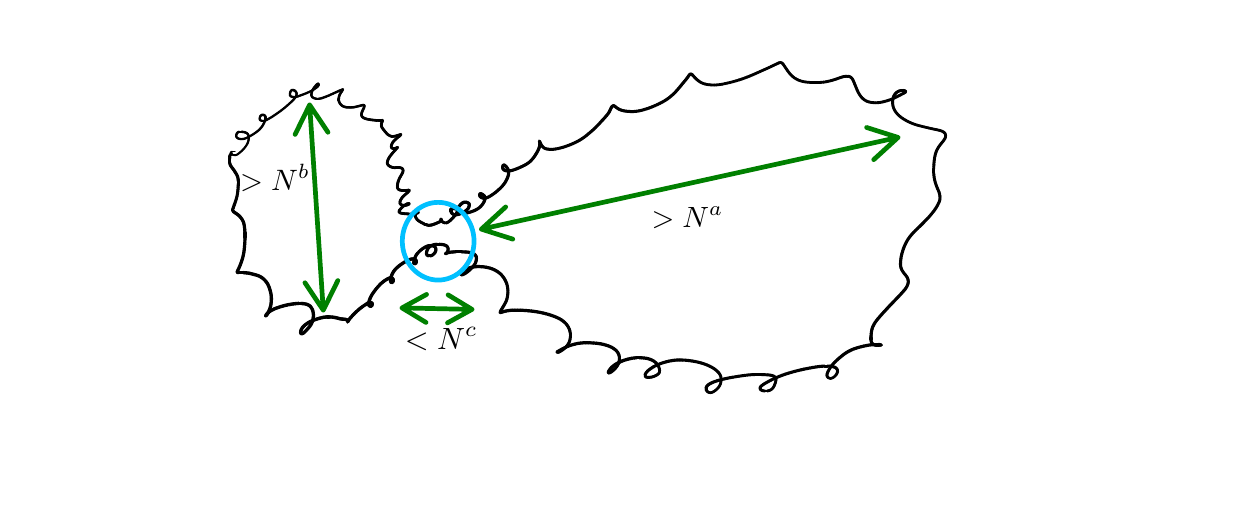}
  \caption{\label{fig:2}An $(a,b,c)$ pinching loop}
  \end{figure}

\begin {proof}
Since the loop soup is a Poisson point process, one just needs to obtain an upper bound on the mass (for the Brownian loop measure $\mu$) of the set of loops in $\Lambda_N$ with this property.
One way to proceed goes as follows: We can cover $\Lambda_N$ with $O(N^{d(1-c)})$ balls of radius $N^c$. Each pair $(x,y)$ as described in the lemma would necessarily be part of one the $O(N^{d(1-c)})$ balls obtained by doubling the radii of each of the previous ones. But for each such ball, the mass of loops that go twice through this ball, with two excursions away from it of diameter at least $N^{a}$ and $N^b$ respectively is bounded by a constant times
$$ (N^c/N^a)^{d-2} \times (N^c / N^b)^{d-2}. $$
One way to see this is to choose to root the loop at the point $z$ located the furthest $L^\infty$-distance from the center of the ball on the loop -- so that the loop is contained in the hypercube centered at the center of the ball, with $z$ on its boundary. So, if $R$ is the side-length of the hypercube, the loop has to first (starting from $z$) to reach the ball while staying in the hypercube (which gives a contribution bounded by a constant times $(1/R) \times (N^c / R)^{d-2}$), then travels to reach the boundary of the ball of radius $N^b$ (which does not contribute to the mass, as this anyway happens for random walk), then back to the ball of radius $N^c$ (which gives a contribution of a constant times $(N^c / N^b)^{d-2}$, and then finally exits the hypercube at $z$ (which contributes to a constant times $1/R^{d-1}$.
So, the mass for the loops corresponding to $z$ is bounded by $R^{-d} [ (N^{2c} / N^{a+b})^{d-2}]$. Summing over all $z$ at distance greater than $N^a$ from the center of the ball, and also on the $O(N^{d(1-c)})$ balls gives the upper bound of the lemma.
\end {proof}

\subsection {Different mesoscopic Brownian loops are in different clusters}

 We consider a critical loop-soup ${\mathcal L}$ in the cable-graph of $[-N,N]^d$ for $d \ge 7$ and look at its clusters.

 \begin {lemma}[No two mesoscopic Brownian loops in the same cluster, see the left part of Figure~\ref {fig:3}]
 \label {Lnotwo}
 Suppose that $a$ and $b$ are in $(0,1)$ with $\nu:= a +b  - ( 1 + 4/(d-2)) $ being positive. Then, the probability of the event ${\mathcal E}_2 (a, b, N)$  that there exists a loop-soup cluster that simultaneously contains a Brownian loop of diameter greater than $N^a$ and another Brownian loop of diameter greater than $N^{b}$ tends to $0$ as $N \to \infty$. It is in fact bounded by a constant times $N^{- \nu}$.
 \end {lemma}
 Note that when $d \ge 7$, then such pairs $a$ and $b$ do exist -- for instance $a=  b = 9/10$ works for all $d \ge 7$.
 We can also note that the larger $d$ is, the larger the range for acceptable $(a, b)$ becomes. Finally (and this is mostly how we will use this result), we can note that when $b > 4 / (d-2)$, one can find $a <1$ is such a way that $\nu >0$.

 \begin {proof}
 The proof is very direct. We will simply upper bound this probability by the expected number of pairs of integer points $\{ x, y \}$ such that (i): $x$ is in a Brownian loop $B_x$ of diameter at least $N^a$, (ii): $y$ is in a (different) Brownian loop $B_y$ of diameter at least $N^b$, and such that (iii): some neighbour of $x$ is connected to some neighbour of $y$ by the clusters of the loop-soup ${\mathcal L}$ with $B_x$ and $B_y$ removed. Indeed, it is straightforward to see that if the event ${\mathcal E}_2 (a,b, N)$ holds, then one can find at least one such pair of points [just take $x$ and $y$ to be the last integer point on the first loop and the first integer point of the second loop on a path in the cluster that joins two integer points in these loops].

 The probability that a given integer point $x$ belongs to a Brownian loop (in the loop-soup) of diameter greater than $N^{a}$ is bounded by a constant times $1/N^{a (d-2)}$ (this is a Poisson random variable, so one just needs to estimate the mass of the loops in $\Lambda_N$ and of diameter greater than $N^a$ that go through $x$, which is upper bounded by the mass of the random walk loops in $\Z^d$ of diameter greater than $N^{a}$ that go through the origin). The individual probability of (ii) is evaluated in the same way, and we have the upper bound for $P[ x \leftrightarrow y ]$.

 By the BK inequality [here, we use the van den Berg-Kesten inequality for Poisson point processes -- that can be derived as a direct consequence from the BK inequality for independent coin tosses, see \cite {vdB} -- alternatively, we could also just use Mecke's equation for Poisson point processes since two of the three events involve only single loops], the probability that (i), (ii) and (iii) hold disjointly (using different Brownian loops in the loop-soup) is therefore upper-bounded by  a constant times
 $$\frac {1}{N^{(d-2)(a+b)}} \times \frac{1}{(|x-y|+ 1 )^{d-2}}.$$
 It now remains to sum over all pairs of points $x$ and $y$ in $[-N, N]^d$, which readily gives an upper bound of a constant times $N^{d+2}/N^{(d-2) (a+b)}$ for
 $P [ {\mathcal E}_2 (a, b, N)]$.
 \end {proof}

\begin{figure}[h]
  \centering
  \includegraphics[width=
  \textwidth]{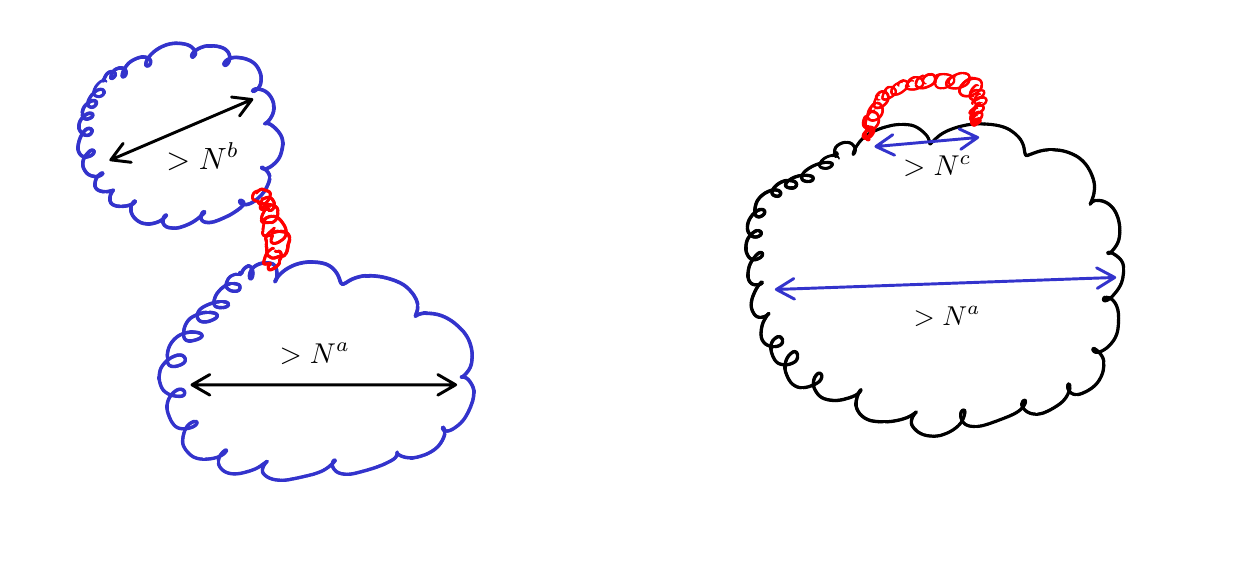}
  \caption{\label{fig:3}Low-probability configurations from Lemma \ref {Lnotwo} (two mesoscopic Brownian loops joined by a chain of loops) and Lemma \ref {L2} (a mesoscopic chain-of-loops-handle on top of a large Brownian loop)}
  \end{figure}

\subsection {No connection between distant points on mesoscopic Brownian loops}
We are going to establish a result about connections between different points of a single loop using the same simple idea:

\begin {lemma}[No connections between distant points on mesoscopic Brownian loops, see the right part of Figure \ref {fig:3}]
\label {L2}
\label {Lo2}
Let $a \ge c$ in $(0,1)$ with $ \nu := (d-2) a + (d-4) c - d >0$. Let ${\mathcal F}(a, c, N)$ be the event that there exist two integer points $x$ and $y$ that are (i) at distance at least $N^{c}$ from each other, (ii) both on the same Brownian loop $B$ of diameter at least $N^{a}$, and (iii) there exists neighbours $x'$ and $y'$ of $x$ and $y$ respectively that belong to the same cluster of of loops for the loop-soup obtained when one removes $B$. Then, the probability of this event is bounded by a constant times $N^{-\nu}$.
\end {lemma}
Note that when $d \ge 7$, such $a$ and $c$ do exist. In particular, if $c > 2 / (d-4)$, it is possible to find $a <1$ so that $\nu >0$.
\begin {proof}
This goes along the very same line as the proof of the previous lemma. For each given $x$ and $y$ (at distance greater than $3$ from each other), the probability that they are on the same Brownian loop of diameter greater than $N^a$ is upper-bounded by a constant times $(1/|x-y|^{(d-2)}) \times (1/N^{a (d-2)})$, while the probability that some neighbours are part of a same cluster of loops is upper-bounded by a constant times $1/|x-y|^{d-2}$. The BK inequality [here, one can also use directly Mecke's equation for Poisson point processes] then shows that the probability that these two events occur disjointly (i.e., using different loops in the loop-soup) is upper-bounded by a constant times $(1/N^{a (d-2)}) \times (1/|x-y|^{2d-4})$. Summing over all pairs of points at distance at least $N^c$ from each other gives the explicit bound in the lemma.
\end {proof}

\subsection {Summarizing}

From now on, we consider some fixed $\beta$ with $\beta \in ( 4 / (d-2), 1)$ as in the theorem. We can then choose $\alpha<1$ sufficiently close to $1$ and $\gamma > 2 / (d-4)$ sufficiently close to $2 / (d-4)$ in such a way that $(a,b,c) = (\alpha, \beta, \gamma)$ satisfies the conditions of Lemmas \ref {L1}, \ref {Lnotwo} and \ref {Lo2} (note also that for any given $\gamma_0 > 2 / (d-4)$, we can choose $\gamma$ to be smaller than $\gamma_0$). These values of $(\alpha, \beta, \gamma)$ will be fixed from now on;  we therefore have that when $N$ is large, with a probability $1 - o(1)$ as $N \to \infty$, for a loop-soup in $[-N,N]^d$:
\begin {itemize}
\item No Brownian loop with diameter at least $N^\alpha$ has a $(\alpha, \beta, \gamma)$ pinching.
\item There is no connection (made by other loops than $B$) between any two points that are at distance greater from $N^{\gamma}$ from each other on any Brownian loop $B$ of diameter greater than $N^\alpha$.
\item No cluster contains simultaneously a loop of diameter greater than $N^\alpha$ and a loop of diameter greater than $N^{\beta}$.
\end {itemize}
Note of course that any loop of diameter greater than some constant $\eps$ times $N$ will have a diameter greater than $N^\alpha$ when $N$ is large.

\section {Proof of the theorem}

Making use of the switching property from \cite {W3}, we will now prove Theorem \ref {thm}. Our proof involves several steps that we now detail.

\subsection {A first use of the switching property}

In this section, we will make a first use of the switching property for loop-soups on the cable-graphs to see that clusters in ${\mathcal C}(\eps, N)$ will contain Brownian loops of diameter at least $\eps N /4$ with conditional probability at least $1/2$ (when one conditions on the cluster). This will allow to derive further results that we will use in the following sections, and serves also as a warm-up to some of the arguments that we will develop in Section \ref {S42}

\subsubsection {Focusing on finitely many tubes suffices}

For each given $\eps$, it is easy to find a finite set $\Delta_1, \ldots, \Delta_{M(\eps)}$ of $(d-2)$-dimensional affine subspaces in $\R^d$ such that any $(d-2)$-dimensional affine subspace $\Delta$ that intersects $[-1,1]^d$ will be very close to at least one of them, meaning that their respective intersections with $[-1,1]^d$ are $\eps/4$ close in Hausdorff distance. In particular, if some cycle in $[-N,N]^d$ has a non-zero index around $N \Delta$ while staying at distance at least $\eps N$ from it, then it will have a non-zero index around some $N \Delta_j$ while staying at distance at least $\eps N / 2$ from it.
So, if ${\mathcal D}_j (\eps, N)$ denotes the set of
clusters in $[-N,N]^d$ that contain a cycle that winds around $N\Delta_j$ and stays at distance at least $\eps N /2$ from it, then
 ${\mathcal C}(\eps, N) \subset \cup_{j \le M(\eps)}{\mathcal D}_j(\eps, N)$.

\subsubsection {The switching property}
Let us briefly recall the version of the switching property from \cite {W3} that we shall use here: Consider a bounded connected subgraph $G$ of the cable graph of $\Z^d$, and a $(d-2)$-dimensional affine subspace $\Delta_0$ of $\R^d$ such that $G \cap \Delta_0 = \emptyset$. We consider a Brownian loop-soup on $G$ and its clusters. A cluster $C$ will therefore either contain cycles with non-zero index around $\Delta_0$ or not. If it does, then the set of possible indices around $\Delta_0$ of oriented cycles in $C$ will be $i(C,\Delta_0) \Z$ for some positive integer $i(C,\Delta_0)$ (because the set of possible indices forms a subgroup of $\Z$).

In a loop-soup, Brownian motions are usually not considered to be oriented. We can however define the index of the unoriented Brownian loop around $\Delta_0$ to be the absolute value of the index of any oriented version of the Brownian loop.

\begin {proposition}[The loop version of the switching property -- Corollary 5 from \cite {W3}]
\label {switch}
For any loop-soup cluster $C$ that contains a cycle around $\Delta_0$, the conditional probability (given $C$) that the sum $\Sigma$ of the indices around $\Delta_0$ of the unoriented Brownian loops in $C$ is an even multiple of $i(C,\Delta_0)$ is exactly $1/2$.
\end {proposition}
Let us just recall that the proof is based on the idea that it is possible to switch in a measure-preserving way the parity of the number of crossings of all edges along any given cycle in $C$ with index $i(C,\Delta_0)$ around $\Delta_0$; this switching then changes the parity of $\Sigma / i(C,\Delta_0)$.

One important trivial observation is that when $\Sigma$ is an odd multiple of $i(C, \Delta_0)$, then there exists at least one  Brownian loop in $C$ with non-zero index around $\Delta_0$.

\subsubsection {Using the switching property in a restricted graph}

We will now  use this in the case where $G \subset [-N,N]^d$, and the cluster $C$ is at distance at least $\eps N/4$ from $\Delta_0$, so that only Brownian loops of diameter greater than $\eps N/2$ can actually contribute to the total index $\Sigma$. We will use the switching property in somewhat similar settings again later in the paper.

\begin {lemma}
For  any $N$ and any cluster $C$ in ${\mathcal C} (\eps, N)$: The conditional probability (given $C$) that this cluster contains a Brownian loop of diameter at least $\eps N/4$ is at least $1/2$.
\end {lemma}

\begin {proof}
It suffices to show that for all $j$, this property holds for all $C \in {\mathcal D}_j (\eps, N)$.
For every $C \in {\mathcal D}_j (\eps, N)$, we can consider separately the following two cases:
\begin {itemize}
\item There is a Brownian loop (in the loop-soup) that is part of $C$ with the property that its projection on the plane orthogonal to $N\Delta_j$ intersects both the circles of radii $\eps N / 4$ and $\eps N/ 2$ around the projection of $N\Delta_j$. Then, this Brownian loop has necessarily a diameter at least $\eps N /4$.
\item If no such Brownian loop exists, then it means that any cycle $\Gamma$ in $C$ that ensures that $C \in {\mathcal D}_j(\eps, N)$ will still be part of a cluster $C'$ of the loop-soup that is obtained by removing all the loops that get to distance smaller than $\eps N/ 4$ from $N\Delta_j$. This is a Brownian loop-soup in the subgraph $G$ of $\Lambda_N$ consisting of all  points that are at distance greater than $\eps N/4$ from $N \Delta_j$. We can therefore apply the switching property (Proposition \ref {switch}) to that loop-soup and that cluster $C'$ to deduce that with conditional probability at least $1/2$, this cluster will contain a Brownian loop with non-zero index around $N\Delta_j$. Such a Brownian loop then necessarily has a diameter at least $\eps N/2$.
\end {itemize}
\end {proof}

\subsubsection {Some direct consequences}

This already immediately implies that the number of clusters in  ${\mathcal C} (\eps, N)$ remains tight (which is part of Theorem \ref {thm}):
\begin {corollary}
\label {co3}
The expected number of clusters in ${\mathcal C} (\eps, N)$ is bounded by twice the expected number of Brownian loops of diameter at least $\eps N /4$.  In particular, for any fixed $\eps >0$, the number of clusters in ${\mathcal C} (\eps , N)$ is tight.
\end{corollary}
In fact, a more quantitative estimate holds: The lemma and the independence of the decomposition of each cluster imply that if $K_\eps(N)$ denotes the cardinality of ${\mathcal C}( \eps, N)$, then the sum of $K_\eps(N)$ independent Bernoulli random variables with parameter $1/2$ is dominated (in law) by the number of Brownian loops of diameter greater than $\eps N /4$, which is a Poisson random variable with a mean that remains bounded  as $N \to \infty$. It therefore follows for instance immediately that the probability that $\sup_N P[K_\eps (N) \ge k]$ decays faster than exponentially in $k$ (as we will see later, $K_\eps$ in fact converges in law to a Poisson random variable).
\medbreak

Another simple direct consequence is the following:
\begin {corollary}
 \label {co4}
The expected number of integer points in the union of all clusters in ${\mathcal C}(\eps, N)$ is bounded by a constant (depending on $\eps$) times $N^4$.
\end {corollary}
\begin {proof}
The lemma shows that this expected number of points is bounded by twice the number of integer points that belong to a cluster that contains a Brownian loop of diameter at least $\eps N / 4$. But the expected number of integer points $x$ that are at distance smaller than $1$ from such a Brownian loop is bounded by a constant times $N^2$. On the other hand, for each given point $x$, the expected number of integer points that are in the same cluster as $x$ is bounded by $N^2$ (by the two-point estimate). We can then use the BK inequality to deduce that the expected number of points in the lemma is bounded by a constant times $N^2 \times N^2$ (indeed a point $y$ would have to be connected to an integer point $x$ neighboring a large Brownian loop via a chain of loops that does not use that large Brownian loop, so that one can obtain the estimate by summing over all $x$).
\end {proof}

\subsection {Statement of the key proposition and consequences}

We start with a Brownian loop-soup in $\Lambda_N$.
We have seen that for each cluster $C$ in ${\mathcal C} (\eps, N)$, the conditional probability (given $C$) that it contains a Brownian loop of diameter at least $\eps N /4 $ is at least $1/2$. Since the expected number of Brownian loops of diameter at least $\eps N /4$ is finite, this makes it possible, for almost each $C$, to define the conditional law of such a Brownian loop of diameter greater than $\eps N / 4$ in $C$.

Let us now sample independently two Brownian loops $B^1(C)$ and $B^2(C)$ according to this conditional law (given $C$) of such large Brownian loops contained in $C$.
When $C$ does contain one or more Brownian loops of diameter at least $\eps N/4$, we can in fact simply choose $B^1(C)$ to be equal to one of them  -- chosen at uniformly at random if there are more than one (or we could not worry about this event because of Lemma \ref {Lnotwo} that ensures that it has small probability). In this way, the conditional probability that $B^1 (C)$ is part of the original loop-soup is therefore at least $1/2$.

We can do this for all clusters in ${\mathcal C} (\eps, N)$ simultaneously (and independently). We can also order the $K_\eps (N)$ clusters of ${\mathcal C}(\eps, N)$ using some deterministic procedure -- for instance by decreasing diameter -- and call them $C_1^N, \ldots, C_{K_\eps(N)}^N$. In the sequel, we will simply write $K_\eps$ instead of $K_\eps (N)$ and $C_i$ instead of $C_i^N$ to simplify notation. It will be implicit that this collection $C_1, \ldots, C_{K(\eps)}$ and its cardinality depend on $N$ (and on the sampling of the loop-soup in $[-N,N]^d$).

The following result will be the key to the theorem (here $d_H$ denotes the Hausdorff distance -- recall also that the constraint on $\beta$ was that $\beta > 4 / (d-2) $).
\begin {proposition}
\label {P7}
One has
$$ \lim_{N \to \infty}  P [ \exists k\le K_\eps , \ d_H (B^1 (C_k), B^2 (C_k)) >  N^{\beta}] = 0. $$
\end {proposition}

Note that this is an ``annealed'' result. The probability in question is averaged over all realizations of the collection $(C_1, \ldots, C_{K_\eps})$. Indeed, in the unlikely case where a cluster was close to be a figure eight type graph, one could guess (and indeed prove using some variation of the switching property) that this Hausdorff distance might be of order $N$ with a sizeable probability (as with a probability bounded from below, the Brownian loop $B^1$ could circumnavigate along one of the two cycles in the figure eight while $B^2$ circumnavigates around the other one).

Let us immediately explain how this proposition leads us much closer to Theorem \ref {thm}:
\begin {itemize}
\item The proposition shows that by sampling $B^2(C_k)$, we have a cycle that is in fact (with high probability) very close to the actual Brownian loop $B^1(C_k)$ contained in $C_k$ (in case $C_k$ does indeed contain such a loop). Since the $B^1 (C_k)$ and $B^2 (C_k)$ are independent (conditionally on $C_k$), one can furthermore note that there exists at least one loop $\Gamma = \Gamma (C_k)$ (and we can in fact choose it as a deterministic function of $C_k$ among all options) with the property that
$$P [ d_H ( B^1 (C_k), \Gamma ) > N^\beta | C_k  ] \le P [d_H(  B^1 (C_k) , B^2 (C_k) ) > N^\beta  | C_k ] .$$
Hence, for all $k_0$,
$$P [K_\eps \le k_0 \hbox { and for all } k \le K_\eps, \  d_H (B^1 (C_k, \Gamma (C_k)) > N^\beta  ] \to 0 $$
as $N \to \infty$.

\item It shows also readily the following very useful fact: For any given $k$, on the event $k \le K_\eps$, the unoriented Brownian loops $B^1 (C_k)$ and $B^2 (C_k)$ will have (with very high probability) the same index around $\Delta$ [indeed, recalling that Brownian loops in the continuum in these dimensions are simple loops, and that these cable graph loops tend to such Brownian loops \cite {LT}]. This shows that (with very high probability) there exists a deterministic function  $I_{C_k}$ of $C_k$ such that the conditional probability that the index of $B^1 (C_k)$ is $I_{C_k}$ is very high [indeed, if two i.i.d. random variables have a probability close to $1$ to be equal, then their law has an atom with mass close to $1$]. In particular, the (annealed, i.e., averaged over all $C_k$'s) probability that this index is an even multiple of $I_{C_k}$ goes to $0$. This therefore implies that the probability that there exists a Brownian loop in the cluster and that the index of this cluster (the sum of all indices of the Brownian loops that are part of the cluster) is an even multiple of $I_{C_k}$ goes to $0$. On the other hand, the switching property does in fact imply that the conditional probability that this cluster index is an even multiple of $I_{C_k}$ is $1/2$.
So, putting these two pieces together, we see that conditionally on $k \le K_\eps$, the probability that $C_k$ contains no macroscopic Brownian loop goes to $1/2$ as $N \to \infty$ [this fact complements the a priori estimate that says that with a probability that goes to $1$, no cluster contains more than one macroscopic loop].
\end {itemize}
So, together with Corollary \ref {co3}, this proves a number of the statements in Theorem \ref {thm} -- the remaining to be proved ones being essentially the ones about the presence/absence of mesoscopic loops that we will derive in Section \ref {S5}.

\medbreak
Note that at the end of the day, it is possible to check that with high probability,  $|I_{C_k}| = |i( C_k, \Delta)|$ for all $k \le K_\eps$ -- indeed, the probability that some large Brownian loop that winds around $\Delta$ is part of a cluster for which $|i(C_k, \Delta)|$ is smaller than the index of the loop goes to $0$ (as it would require an extra-connection outside of $\beta$ of the type that will be excluded in the course of the proof).

\subsection {Proof of the proposition}
\label {S42}
This is probably the most intricate proof in this paper. We will decompose it into several steps:
We want to show that for all fixed $\eps$ and $\eta$, if we set
$$ {\mathcal M} (N) := \{ \exists k\le K_\eps , \ d_H (B^1 (C_k), B^2 (C_k)) >  N^{\beta} \},$$
then when $N$ is large enough,
$P  [{\mathcal M}(N) ] \le \eta$.

\subsubsection* {Step 1: Reduction using the previous estimates}
A first remark is that  since
$P [ K_\eps  (N) \ge k ] \le E[ K_\eps (N) ] / k \le (\sup_N E [ K_\eps (N)]) / k$, we use Corollary \ref {co3} and  choose $k _0 =k_0 (\eta)$ in such a way that for all $N$,
$P [ K_\eps (N) \ge k_0 ] \le \eta/ 100$. So, we do not really to worry about the configurations with many large clusters.

When $k \le K_\eps$, we write $B^1_k = B^1 (C_k)$ and $B_k^2 = B^2 (C_k)$.
We know that the probability that $B^1_k$ is really part of the original Brownian loop-soup is at least $1/2$ (for each $k$, and independently of $C_k$). Hence, one can view the collection $(B_k^1)_{k \le K_\eps}$ as a subset of the union of two (correlated) Poisson point process of Brownian loops in $\Lambda_N$ of diameter at least $\eps N /4$. This allows to use a priori results about the properties of these Brownian loops.
The path $B^2_k$ is only ``virtually'' a Brownian loop in a loop-soup, but since by definition, $(B_k^1)_{k \le K_\eps}$ and $(B_k^2)_{k \le K_\eps}$ have the same law, this second collection has the same a priori properties.

In particular, we can apply some of the preliminary estimates to these collections of loops. For instance, the probability of the event
$\tilde {\mathcal E}_1 ( \alpha, \beta, \gamma)$ that at least one of the loops $(B_k^2)_{k \le K_\eps}$ has a $(\alpha, \beta, \gamma)$ pinching as in Lemma~\ref {L1} is bounded by twice the probability of ${\mathcal E}_1 (\alpha, \beta, \gamma)$ in Lemma~\ref {L1} and therefore goes to $0$ as $N \to \infty$. Similarly, the probabilities of the events $\tilde {\mathcal E}_2 (\alpha, \beta, N)$ and  $\tilde {\mathcal F} (\alpha, \gamma, N)$ defined analogously but this time for the family $(B_k^1)_{k \le K_\eps}$  and in reference of the events in Lemmas \ref {Lnotwo} and \ref {L2}, do go to $0$ as $N \to \infty$.

Hence, it is in fact sufficient to prove that for all $N$ large enough
$$
P [ {\mathcal M} (N) \cap \{ K_\eps(N) \le k_0 \} \cap \tilde {\mathcal E}_1^c (\alpha, \beta, \gamma, N) \cap  \tilde {\mathcal E}_2^c (\alpha, \beta,N) \cap \tilde {\mathcal F}^c (\alpha, \gamma, N)] < \eta/2$$
[the superscript $c$ means that we are considering the complements of the events].
Let us denote this intersection event by ${\mathcal N}(N)$.

\subsubsection* {Step 2: A further a priori feature of $B_k^1$ and $B_k^2$}

Let $\Pi_0$ denote the plane $\R^2 \times \{ 0 \}^{d-2}$, and let $\pi_0$ be the orthogonal projection on $\Pi_0$ (we will use this both in the continuum and for discrete lattices). For each positive $\xi$, we can find a deterministic finite collection of points $(z_i)_{ i \le J}$ (where $J = J(\xi)$) on $\Pi_0$ so that any point in $\Pi_0 \cap \pi_0 ( [-1,1]^d)$
is at distance less than $\xi /4$ from at least one of the $z_i$'s. For each $N$, we then define $y_i := N z_i$ (so that each point in $[-N, N]^d$ is at distance at most $\xi N /4$ of at least one of the $\Delta_i^N := \pi_0^{-1} (y_i)$). We then define the event $A_1 ( \eps, \xi, i, N)$ that the loop $B_k^1$ does wind around $\Delta_i^N$ while staying at distance at least $\xi N$ from it (and the similar event $A_2 (\eps, \xi, i, N)$ for the loop $B_k^2$).

\begin {lemma}
\label {proof1}
For every $\eps$ and $\eta$, one can find $\xi$ small enough such that for all $N$ large enough,
$$ P [ \cup_{i \le J} A_1 (\eps, \xi, i, N ) ]  \ge 1 - \eta /20  .$$
\end {lemma}
In other words and loosely speaking, this means that with high probability, each of the $\pi_0 (B_k^1)$ and $\pi_0 (B_k^2)$ will wind around one of the finitely many
``$\xi N$-macroscopic holes''.

\begin {proof}
This simply follows from the convergence of the (random finite collection) of loops $(B_k^1, k \le \min ( K_\eps, k_0) )$ to the corresponding collection of continuum Brownian loops (using the uniform convergence), which is for instance established in \cite {LT}. The points $y_i$ correspond to $N z_i$ where the $z_i$'s are chosen in such a way that with probability at least $1- \eta /100$, any continuum Brownian loop of diameter greater than $\eps /2$ in a loop-soup in $[-1, 1]^d$ does wind around at least one $\Delta_{z}$ while staying at distance at least $2 \xi$ from it (and therefore winds around of the $\Delta_{z_i}$'s while staying at distance at least $\xi$ from it).
\end {proof}

Through the coming sections (with $\eps$ and $\eta$ fixed), we choose $\xi$ as in this lemma. In view of these first steps, in order to prove the proposition, it is
in fact sufficient to prove that (still for each fixed $\eps$, $\eta$, $\xi$) for each $i_1, i_2 \le J(\xi)$,
the probability of
$$ P [ {\mathcal N}(N)  \cap  A_1 (\eps, \xi, i_1, N) \cap A_2 (\eps , \xi, i_2, N) ] \to 0  $$
as $N \to \infty$.
Indeed, one can then sum this over the $J(\xi)^2$ values of $(i_1, i_2)$ to see that
$$ P [ {\mathcal N}(N) \cap ( \cup_{i \le J} A_1 (\eps, \xi, i, N ) ) \cap ( \cup_{i \le J} A_1 (\eps, \xi, i, N ) ) ] \to 0, $$
so that one indeed has
$$ P [ {\mathcal N}(N) ] \le \eta/2$$
for all $N$ large enough.
\medbreak
 In the coming two subsections, we will fix $i_1$ and $i_2$ and simplify notations by letting (the dependence in $N$ being implicit) $\Delta_1:=N \Delta_{i_1}^0$ and $\Delta_2 := N \Delta_{i_2}^0$, $A_1 =  A_1 (\eps, \xi, i_1, N)$ and $A_2 =   A_2 (\eps , \xi, i_2, N)$. The goal is therefore to prove that $P[ {\mathcal N} \cap A_1 \cap A_2 ]$ goes to $0$ as $N$ tends to infinity.

\subsubsection* {Step 3: $B_k^1$ and large cycles in $B_k^2$ have to intersect}

Our next step is to combine Lemma \ref {Lnotwo} with the switching property to see that $B_k^1$ and any large cycle contained in $B_k^2$ will intersect (with high probability).
More precisely, we will prove the following stronger statement:
\begin {lemma}
\label {L55}
Let ${\mathcal H}(N)$ be the event that for all $k \le K_\eps$,  the loop-soup obtained with just $B_k^1$ removed does contain no cluster that does (a) intersect the trace of $B_k^1$ and (b) also contain a cycle around $\Delta_2$ while staying at distance at least $\xi N $ from it. Then $P[{\mathcal H}(N)] \to 0$ as $N \to \infty$.
\end {lemma}

\begin {proof}
The first idea is (for each given $\eps$ and each loop-soup in $\Lambda_N$) to partially resample the Brownian loops of diameter greater than $\eps N / 4$ by (a) removing each of them independently with probability $1/2$ and (b) sampling another Poisson point process of Brownian loops with diameter greater than $\eps N /4$ with half the intensity  -- intuitively, one resamples ``half'' of these large loops. The new obtained loop-soup is then also distributed as a Brownian loop-soup to which we can also apply the switching property. Furthermore, when one resamples, the probability that one removes exactly one Brownian loop (here $B_k^1$) and adds no Brownian loop of diameter greater than $\eps N/4$ from the new loop-soup is bounded from below by $2^{-k_0-1}$ (when conditioned on the event that $K_\eps \le k_0$).

Then, on this event, if we observe a cycle $\Gamma$ around $\Delta_2$ in one of the new clusters (created by removing $B_k^1$) that winds around $\Delta_2$ while staying at distance $\xi N$ from it, then we  have the similar dichotomy as before:
\begin {itemize}
 \item Either there initially was a Brownian loop intersecting $\Gamma$ and getting $\xi N/2$ close to $\Delta_2$. In that case, it means that there exists a cluster $C_k$ in the initial loop-soup that contains both $B_k^1$ and another Brownian loop of diameter at least $\xi N/2$.
 \item Or if not, then we can apply the switching property to the cluster of Brownian loops (that still contains $\Gamma$) obtained when removing $B_k^1$ as well as all the loops that get to distance smaller than $\xi N/2$ from $\Delta_2$ and conclude that with positive conditional probability, this remaining loop-soup contains a Brownian loop that winds around $\Delta_2$, that therefore also has a diameter at least $\xi  N$.
\end {itemize}
So, altogether, the probability in Lemma \ref {L55} is bounded by a constant times the probability that for some $k$, the cluster that contains $B_k^1$ also contains a second Brownian loop with diameter at least $\xi N/2$, which allows to conclude thanks to Lemma \ref {Lnotwo}.
\end {proof}

Note that this lemma shows in particular that the probability that for some $k \le K_\eps$, $B_k^1 \cap B_k^2 = \emptyset$  while $A_1$ and $A_2$ hold  goes to $0$. Indeed, if this was the case, then the trace of $B_k^2$ would contain a cycle satisfying condition (b).

\subsubsection* {Step 4: No long excursion of $B_k^2$ away from $B_k^1$}

We will conclude the proof of the proposition in this section by showing that the probability of ${\mathcal N}(N) \setminus {\mathcal H}(N)$ tends to $0$.

Let us first see how to combine the various inputs in the definition of this event:

\begin {itemize}
\item We note that on this event (as we just pointed out in the previous step), the two loops $B_k^1$ and $B_k^2$ will necessarily intersect. Therefore, a point on $B^2_k$ that is at distance greater than $N^\beta$ from $B^1_k$ will necessarily belong to some excursion $e$ of $B^2_k$ away from $B^1_k$.  On the event $\tilde {\mathcal F}(\alpha, \gamma, N)$, this implies that the two endpoints $u$ and $v$ of this excursion are not more than $N^{\gamma}$ apart.

\item
Furthermore, if $\tilde {\mathcal E}_1(\alpha, \beta, \gamma)$ does not hold, then none of the loops $B_k^2$ for $k \le K_\eps$ has an $(\alpha, \beta, \gamma)$-pinching.  So, if the excursion $e$ of $B_k^2$ away from $B_k^1$ has diameter greater than $N^\beta$ (which necessarily is the case if it contains a point that is at distance greater than $N^\beta$ from $B_k^1$) and has endpoints that are less than $N^\gamma$ apart, then it means that $B_2^k \setminus e$ (viewed as a continuous path from $u$ to $v$) has diameter smaller than $N^\alpha$. But since $B_k^1$ has diameter at least $\eps N /2$ by definition, this means that $e$ is in fact almost the entire loop $B_k^2$. But we can now use the no-pinching property again, to deduce that the diameter of $B_k^2 \setminus e$ is smaller than $N^\beta$. In particular, this shows that the diameter of $B_k^1 \cap B_k^2$ is smaller than $N^\beta$. We are therefore looking at a configuration as in Figure \ref {fig:4}.
\end {itemize}

\begin{figure}[h]
  \centering
  \includegraphics[width=
  \textwidth]{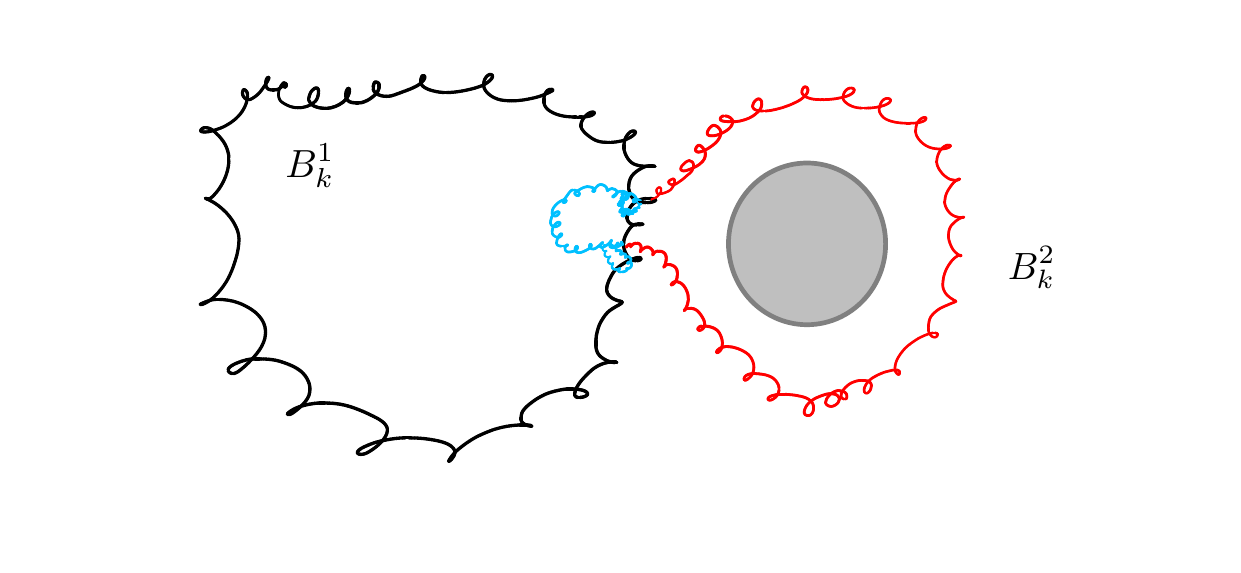}
  \caption{\label{fig:4}The remaining case to exclude: The excursion of $B_k^2$ away from $B_k^1$ is almost all of $B_k^2$}
  \end{figure}

The goal is therefore to show that this last unlikely looking scenario has indeed a very small probability.
Here, we can not use directly the very same trick where one removes $B_k^1$ from the loop-soup (by partial resampling) and directly applies the switching property to the remaining cluster, because that remaining cluster will not necessarily contain a large cycle anymore (indeed, $B_k^2$ could have ``used'' some part of the trace of $B_k^1$ that disappeared), so some further argument is required.

We have just argued that if ${\mathcal N}(N) \cap {\mathcal H}_N$ holds,
there exist two $N^{\gamma}$-close integer points $u$ and $v$ on (the $1$-neighborhood of) $B^1_k$, for which there exists a chain of loops (all of size smaller than $N^{\beta}$ -- this is because $\tilde {\mathcal E}_2 (\alpha, \beta)$ holds) in the loop-soup that join $u$ and $v$ and (if one adds the straight line from $u$ to $v$ to it) creates a cycle that has a non-trivial winding around  $\Delta_2$  while staying at distance greater than $\xi N/ 2$ from it.

One first key observation is that it is enough to focus on the case where the ordered chain of intersecting loops is a finite chain of loops that ``uses'' each Brownian loop only once along the way. Indeed, if a Brownian loop $B$ appears twice in the chain then either one can remove the part of the chain of loops between these two occurrences by just using the loop $B$ itself, or if this shortcut did create a much smaller cycle that does not wind around $\Delta_2$ anymore, it implies that there is a large cycle (around $\Delta_2$) in the cluster with $B_k^1$ removed, which contradicts ${\mathcal H}_N$ (see Figure \ref {fig:5}).

\begin{figure}[h]
  \centering
  \includegraphics[width=
  \textwidth]{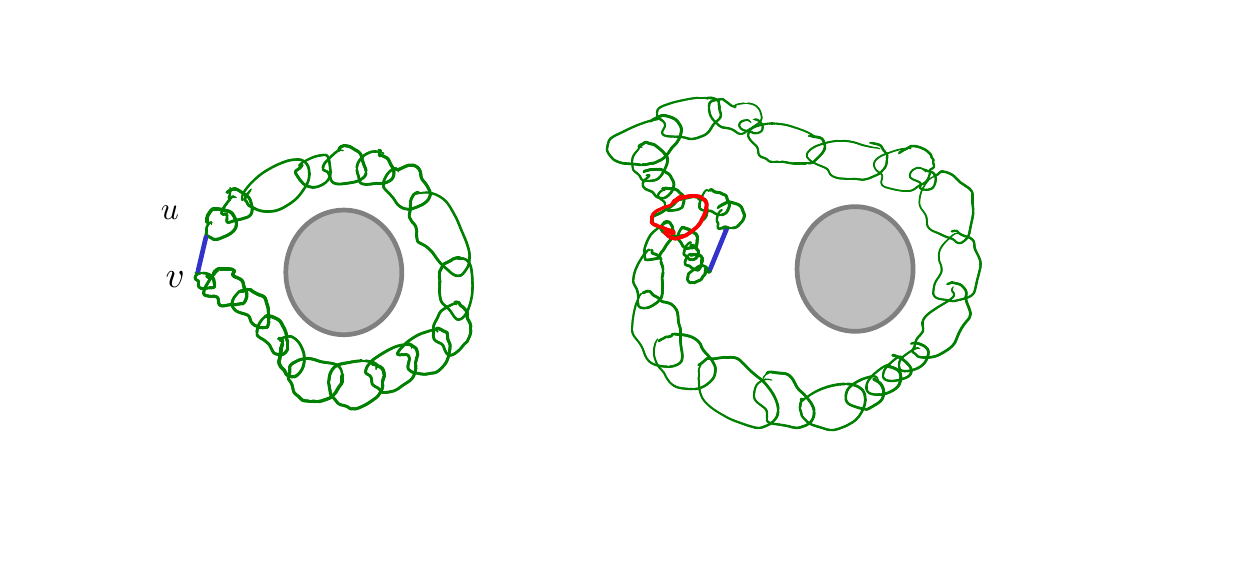}
  \caption{\label{fig:5}The chain of small/mesoscopic loops from $u$ to $v$ -- if one loops is used twice at different windings (here in bold/red), there exists a chain of loops surrounding $\Delta_2$ (right)}
  \end{figure}

A second observation is that this chain of loops can then be lifted to the universal cover (around $\Delta_2$) of $\Z^d \setminus \Delta_2$ in such a way that it contains a path from some fixed lift $u_0$ of $u$ to one of the lifts $v_i$ of $v$, where this lift is not allowed to be the lift $v_0$ that is $N^{\gamma}$-close to the starting point $u_0$.

For each fixed $u$ and $v$, we will now compare the probability of the previous event  with the probability of connections in the universal cover to  get an upper bound given by a constant (that depends on $\eps$ and $\xi$) times $1/ N^{d-2}$.

One way to proceed is the following: We start from $u$, and discover the loops of the cluster that contains $u$ (in the loop-soup obtained by keeping only loops of diameter at most $N^{\beta}$ that stay at distance $\xi N / 2$ from $\Delta_2$) iteratively, in a way that will allow to couple it with a loop-soup in the universal cover:
\begin {itemize}
 \item We first discover all the loops that go through $u$. This is a Poisson point process of loops, and the union of these loops is called $U_1$ and has diameter at most $2 N^{\beta}$.
 \item We then discover the loops that intersect $U_1$ but did not intersect $u$. The union $U_2$ of these loops with $U_1$ has then a diameter at most $4 N^{\beta}$.
 \item We continue iteratively
\end {itemize}
Given that at each step, one adds only loops of diameter smaller than $N^{\beta}$, one can associate to each discovered loop the winding number (from $u$) at the time of the discovery of that loop. Indeed, if the loop was discovered at the same step on two different sheets of the universal cover, then it implies that the cluster contains a cycle around $\Delta_2$, which is forbidden by the event ${\mathcal H}_N$.

This implies that, when restricted to the event that we are interested in, we can view the discovered loop as a subset of the loops in the universal cover (around $\Delta_2$) and that the connection from $u$ to $v$ that we are looking for is bounded by the probability a (non-restricted) loop soup in the universal cover of $\Z^d$ with the $\xi N /2$ neighborhood removed connects $u_0$ to at least one of the lifts $v_i$ of $v$ that are at distance at least $\xi N/2$ from $u_0$, which is bounded in terms of the Green's functions $G(u_0,v_i)$ in this universal cover. By summing over $i$, one readily sees that this is upper bounded by some constant times $1 / N^{d-2}$.

Summing over all $u$ and $v$ (recall that we ask $u$ and $v$ to be part of a loop of size of order $N$ as well that are at distance less than $N^{\gamma}$ from each other -- the sum of these probabilities is bounded by $N^{2 + 2 \gamma}$), one concludes that this probability is bounded by
a constant times $ N^{2 + 2 \gamma - (d -2 )}$ that  does indeed tend polynomially fast to $0$ as $N \to \infty$.

\subsection {About the mesoscopic loops in cycle-containing clusters}
\label {S5}

Finally, we complete the proof of the theorem by discussing the presence of mesoscopic Brownian loops in the clusters in ${\mathcal C}(\eps, N)$.

This section is divided into three quite different steps all describing features that occur with a probability that goes to $1$ as $N\to \infty$: We will first argue that in fact no cluster in ${\mathcal C}(\eps, N)$ will contain a Brownian loop with diameter larger than $\eps N /10$ that does not belong to ${\mathcal B}(\eps, N)$. A by-product of this fact will be that most clusters in ${\mathcal C}(\eps, N)$ will in fact be in ${\mathcal C}(\eps', N)$ for $\eps'$ a little bit bigger than $\eps$.  The second step will be to show that clusters in ${\mathcal C} (\eps, N)$ that contain no Brownian loop of diameter greater than $\eps N / 10$ do contain large cycles that are created by the loops of diameter smaller than $N^{\gamma_0}$ only (for any fixed $\gamma_0$ larger than $2 / (d-4)$). In the final step, we will explain that clusters in ${\mathcal C}(\eps, N)$ will contain no loop of diameter greater than $N^\beta$ apart from those loops that are in ${\mathcal B}(\eps, N)$ (recall that $\beta$ has been chosen to be fixed and larger than  $ 4 / (d-2)$).

\subsubsection* {Step 1: The clusters than contain large loops}
We first focus on the clusters that contain Brownian loops of diameter greater than $\eps N / 10$.
\begin {lemma}
\label {L51}
The probability that some cluster  in ${\mathcal C}(\eps, N)$ contains a Brownian loop of diameter greater than $\eps N /10$  that is not in ${\mathcal B} (\eps, N)$ tends to $0$.
\end {lemma}
In other words, the probability that a large Brownian loop needs and uses ``extra help'' from the other loops in order to help its cluster to pass the threshold for being in ${\mathcal C}(\eps, N)$ goes to $0$. Note that we will upgrade this result in Step 3 (we will there in fact show that one can replace $\eps N /10$ by $N^\beta$ in the statement of the lemma).

\begin {proof}
Recall that the collection of Brownian loops of diameter greater than $\eps N / 10$ in the loop-soup is tight. If $B$ is such a loop, $C$ is its cluster and $\Gamma$ is a cycle in $C$ that ensures that $C \in {\mathcal C}(\eps, N)$ [i.e., this cycle will wind around some $\Delta$ while staying at distance at least $\eps N$ from it], then:

First, using the same arguments as before, we can observe that with high probability, $\Gamma$ and $B$ have to intersect. Indeed, we can resample the large Brownian loops, then (as before) make only $B$ disappear (with conditional probability bounded from below), and note that if $\Gamma \cap B = \emptyset$, the cluster containing $\Gamma$ for this resampled loop-soup would still be in ${\mathcal C}(\eps, N)$, and therefore contain a large Brownian loop with conditional probability $1/2$, which in turn would again contradict Lemma \ref {Lnotwo} for the initial loop-soup.

Next, we know by Lemma \ref {L2} that (with high probability), any excursion away from $B$ by $\Gamma$ has endpoints $u$ and $v$ that are less than $N^c$ apart for any fixed $c > 2/ (d -4)$. The probability that one of these excursions away from $B$ does more than a half-winding around some $\Delta$ is also bounded by the same argument as above (for this, one can first consider a finite set of appropriate $\Delta_i$'s).  On the other hand, $B$ will have no $(a,b,c)$ pinching points (with high probability, for well chosen $a>b$ in $(c,1)$), so that at least one of the parts of $B$ joining $u$ to $v$ has diameter smaller than $N^b$, and will therefore have the same (small) winding around $\Delta$ than the excursion by $\Gamma$ from $u$ to $v$. This indicates that the Brownian loop $B$ will actually contain a cycle around  $\Delta$ (one just follows $\Gamma$ and replaces each excursions of $\Gamma$ away from $B$ by the corresponding portions of $B$ of diameter smaller than $N^b$). Since $\Gamma$ is at distance at least $\eps N$ from some $\Delta$, it also follows that (with high probability), this cycle contained in $B$ will be at distance at least $\eps N - N^b$ from this same $\Delta$.

\begin{figure}[h]
  \centering
  \includegraphics[width=
  \textwidth,angle=180]{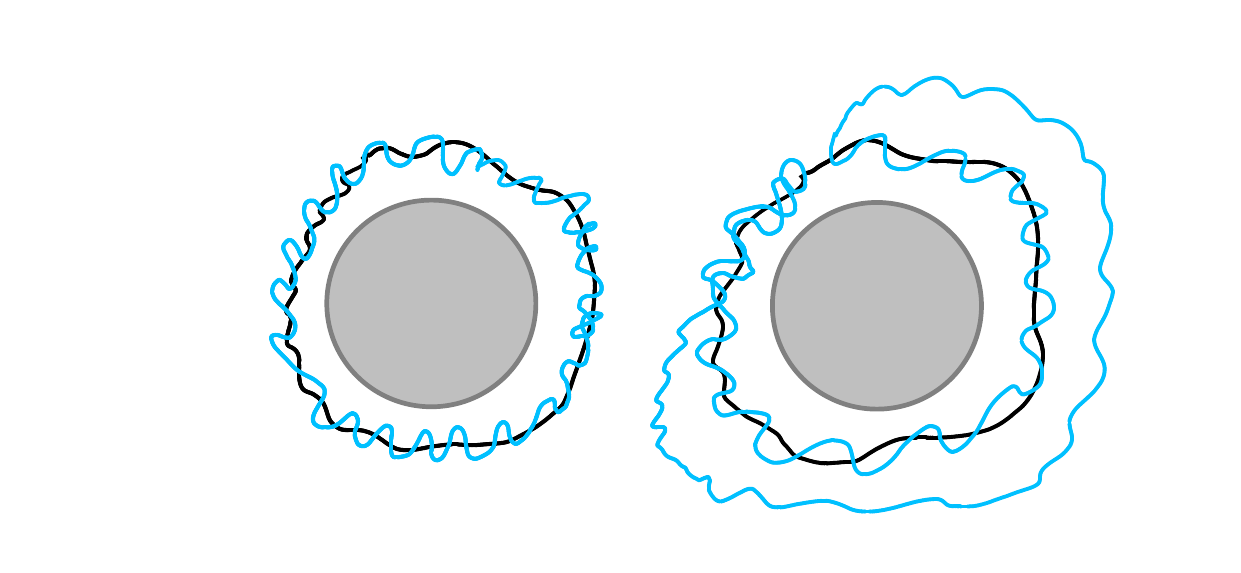}
  \caption{\label{fig:50}If $\Gamma$ (in dark) has no long excursion away from $B$ then either $B$ has a pinching point (left) or wraps around $\Delta$ (right)}
  \end{figure}

So, on the one hand, $B$ is not in ${\mathcal B}(\eps, N)$ so that it is at distance less than $\eps N$ from any $\Delta$ that it winds around, and on the other hand, we have just argued that (with high conditional probability), it contains a cycle at distance greater that $\eps N - N^b$ from one $\Delta$ that it wraps around. The probability of this happening for some loop in the loop-soup goes to $0$ as $N \to \infty$, because it can be asymptotically compared to the probability that the maximum (over $\Delta$'s that it winds around) distance of a continuum Brownian loop in $[-1,1]^d$  being exactly equal to $\eps$ (for instance using the convergence in \cite {LT}), which can be easily seen to be equal to $0$ (for instance using scale-invariance of the Brownian loop measure).
\end {proof}

Let us state one direct consequence of the previous proof as a separate statement:
\begin {corollary}
\label {coprime}
For any fixed $\eps$, the probability that ${\mathcal C}(\eps, N) \setminus {\mathcal C}(\eps', N)$ is empty goes to $1$ as $\eps' \to  \eps+$ independently of $N$ (i.e., the sup over $N$ of these probabilities goes to $0$).
\end {corollary}

Another useful consequence of this lemma is that when a cluster in ${\mathcal C}(\eps, N)$ does not contain a Brownian loop in ${\mathcal B}(\eps, N)$, the cycles that will ensure that it belongs to ${\mathcal C}(\eps, N)$ will necessarily involve chains of many loops of diameter much smaller than $N$. Indeed, by Lemma \ref {Lnotwo}, it can typically not contain more than one loop of diameter greater than $N^{9/10}$, and if the largest loop in the cluster has diameter smaller than $\eps N / 10$, the chain (which has diameter at least $\eps N$) will have to involve at least $9\eps N^{1/10}/10$ of the other loops with diameter smaller than $N^{9/10}$.

\subsubsection* {Step 2: Some features of clusters in ${\mathcal C}(\eps, N)$ that contain no macroscopic loops}

We will now focus on the clusters in ${\mathcal C}(\eps, N)$ that contain no Brownian loop of diameter greater than $\eps N /10$. In that case, since Lemma \ref {Lnotwo} indicates that it can contain no more than one Brownian loop of diameter greater than $N^{9/10}$, it implies a large number of loops have to be involved in creating a cycle that ensures that $C \in {\mathcal C}(\eps, N)$.

Our first result here is the following:
\begin {lemma}
The probability that some cluster in ${\mathcal C}(\eps, N)$ contains (a) No Brownian loop of diameter greater than $\eps N /10$ and (b) no cycle ensuring that the cluster is in ${\mathcal C}(\eps, N)$ that is only created by loops of diameter smaller than $N^{9/10}$ goes to $0$.
\end {lemma}
Combining this with Lemma \ref {L51}, this implies that (with a probability that goes to $1$ as $N \to \infty$) the clusters in ${\mathcal C}(\eps, N)$ that contain no loop in ${\mathcal B}(\eps, N)$ will contain large cycles created by  loops of diameter smaller than $N^{9/10}$ alone.
\begin {proof}
By Corollary \ref {coprime}, we see that it is sufficient to bound the probability that there exists a cluster in ${\mathcal C}(\eps', N)$ for which  (a) and (b) happens goes to $0$ (for any fixed $\eps' > \eps$).
By Lemma \ref {Lnotwo}, with a probability that goes to $1$, no cluster will contain more than one loop of diameter greater than $N^{9/10}$. Suppose that a cluster $C$ contains a cycle $\Gamma$ around some $\Delta$ that does stay at distance $\eps' N$ from it. If some loop $B$ of diameter in $[N^{9/10}, \eps N /10]$ intersects it, then $\Gamma$ will contain a ``long excursion'' away from $B$ that makes more than a half-turn around $\Delta$ and that is made with loops of diameter smaller than $N^{9/10}$. There are two options, just as in the proof of the main proposition:

\begin{figure}[h]
  \centering
  \includegraphics[width=
  \textwidth]{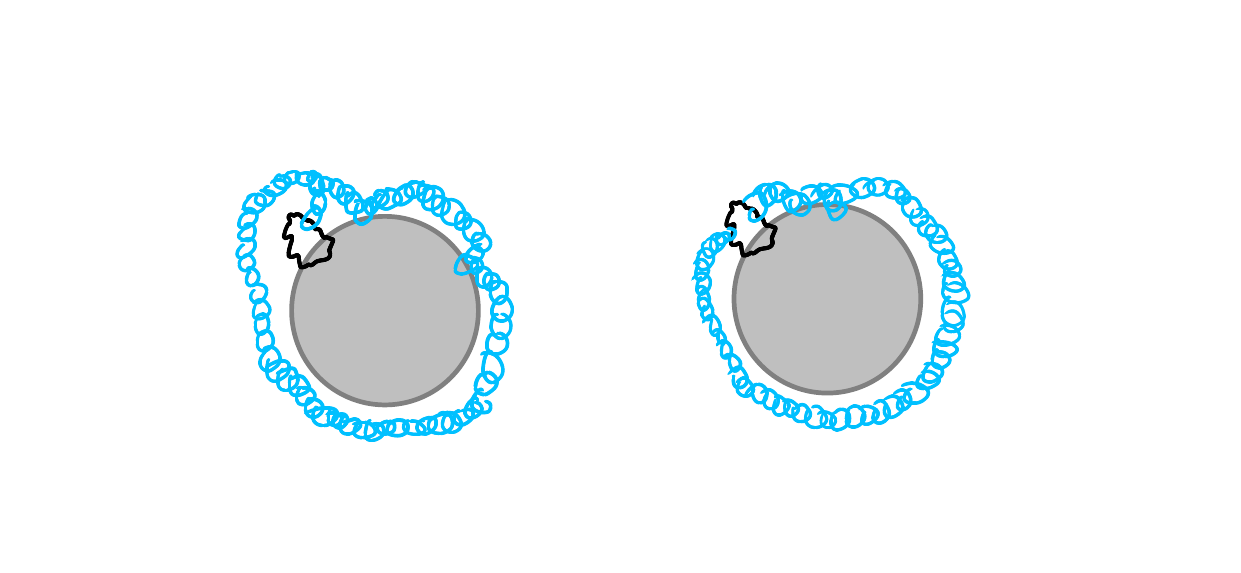}
  \caption{\label{fig:451}Either the chain contain itself a cycle around $\Delta$ (left) or it is a chain in the universal cover (right)}
  \end{figure}
\begin {itemize}
\item Either it uses the same Brownian loop twice at different winding numbers (i.e. for different lifts on the universal cover) -- in that case the chain of loops involved in the excursion will actually contain a cycle around $\Delta$ that stays at distance larger than $\eps' N- N^{9/10}$ (which is larger than $\eps N$ when $N$ is large) from it. So this chain of loops of diameter smaller from $N^{9/10}$ will ensure that the cluster is in ${\mathcal C}(\eps, N)$.
 \item Or one can modify the excursion in such a way that it uses each loop once along the chain. By the same argument with the universal cover (here, we can first use a given finite family of $\Delta_j$'s so that any cycle will wrap around one of them), we can upper bound the probability of this occurring by a constant over $N^{d-2}$ times the expected number of pairs of integer points on loops of diameter greater than $N^{g}$ in the loop for $g=9/10$. We get the upper bound of a constant times
 $$ N^{d(1-g)} N^{4g} / N^{d-2} = N^{2} / N^{g (d-4)}$$
 which indeed goes to $0$ as $N \to \infty$ as $g > 2 / (d-4)$.
\end {itemize}
\end {proof}

We can note that the $N^{9/10}$ threshold was used only to ensure that the cycle $C$ was using no loop of size greater than $N^{9/10}$, while the final estimate did only use the fact that $9/10$ is greater than $2 / (d-4)$. But now, the lemma itself does ensure that the cycle can anyway be chosen in such a way to use no Brownian loop of diameter greater than $N^{9/10}$.
This will enable us to  upgrade the result as follows (valid for any fixed value of $\gamma_0$ greater than $2 / (d -4)$):

\begin {lemma}
 \label {betterL}
The probability that some cluster in ${\mathcal C}(\eps, N)$ contains (a) No Brownian loop of diameter greater than $\eps N /10$ and (b) no cycle ensuring that the cluster is in ${\mathcal C}(\eps, N)$ that is only created by loops of diameter smaller than $N^{\gamma_0}$ goes to $0$.
\end {lemma}
\begin {proof}
As in the previous proof, it is sufficient to focus on the clusters in ${\mathcal C}(\eps' , N )$.
Let $\Gamma$ be a cycle of a cluster $C$ that ensures that $C$ is in ${\mathcal C}(\eps' , N )$, and such that $C$ contains no loop of diameter greater than $\eps N /10$. By our previous lemma, by discarding an event of probability that goes to $0$ as $N \to \infty$, we can further assume that the cycle does not intersect any loop of diameter greater than $N^{9/10}$.
In particular, all the loops involved in creating this chain are at distance greater than $\eps N$ from some $\Delta$ that the cycle winds around.

It is then possible to choose a minimal subset $B_1, \ldots, B_n$ of these loops in such a way that their union contains a cycle around $\Delta$ while no proper subset of this collection of loops would contain such a cycle. In other words, the loops will form a circular chain around $\Delta$ with each $B_j$ intersecting only $B_{j-1}$ and $B_{j+1}$ (with the convention $B_{n+k}=B_{k}$).
Suppose that $B_1$ is the Brownian loop with the largest diameter in this cycle. Then it means that the other Brownian loops will contain a contain a connection from an integer point neighboring $B_1$ to another integer point neighboring $B_1$. We can then (using again the same argument with the universal cover) upper bound the probability that the diameter of $B_1$ is greater than $N^{\gamma_0}$ by a constant over $N^{d-2}$ times the expected number of pairs of points on some Brownian loop with diameter greater than $N^{\gamma_0}$. This is exactly the same bound as in the end of the proof of the previous lemma -- and since $\gamma_0 > 2/ (d-4)$, it indeed goes to $0$.
\end {proof}

\subsubsection* {Step 3: Wrapping up}
We have now seen that (all this with probability that goes to $1$ as $N \to \infty$), the $K_\eps$ clusters in $C \in {\mathcal C} (\eps, N)$ can be split into two subfamilies (here $\beta > 4 / (d-2)$ and $\gamma_0 > 2 / (d-4)$ are arbitrary but fixed):
\begin {itemize}
 \item The clusters than contain a Brownian loop in ${\mathcal B}(\eps, N)$. Note that these clusters will not contain any Brownian loop of diameter greater than $N^\beta$ by Lemma \ref {Lnotwo}.
 \item The clusters that contain no Brownian loop of diameter greater than $\eps N / 8$ and that contain a set of loops of diameter smaller than $N^{\gamma_0}$ that creates a cycle that winds around some $\Delta$ while staying at distance greater than $\eps N$ from it.
\end {itemize}
Since the number of Brownian loops in ${\mathcal B}(\eps, N)$ converges to a Poisson random variable (with mean given by the mass of Brownian loops of the corresponding set of loops in $[-1,1]^d$), it follows that $K_\eps$ converges to a Poisson random variable with twice this mean.
\medbreak

To wrap up, let us now argue that the latter clusters will in fact contain no Brownian loop of diameter greater than $N^\beta$. In other words:
\begin {lemma}
With a probability that goes to $1$ as $N \to \infty$, no cluster in ${\mathcal C}(\eps, N)$ contains a Brownian loop with diameter greater than $N^\beta$ that is not in ${\mathcal B}(\eps, N)$.
\end {lemma}
\begin {proof}
By the previous results (and Lemma \ref {Lnotwo} in particular), we only need to focus on the clusters that contain no loop of diameter greater than $\eps N /8$.
It is here useful to recall from Corollary \ref {co4} that the expected number of integer points that do belong to clusters in ${\mathcal C}(\eps, N)$ is bounded by a constant (that depends on $\eps$) times $N^4$.
In particular, if we first only sample the Brownian loops of diameter smaller than $N^\beta$ in the loop-soup (so this is a restricted loop-soup ${\tilde L}$), the number of integer points that belong to a cluster $\tilde C$ that is already in ${\mathcal C}(\eps, N)$ is also upper bounded by the same quantity (as those points will anyway end up in a cluster in ${\mathcal C}(\eps, N)$ when one adds more Brownian loops). Let $U$ denote the set of points that are at distance not greater than $1$ from the union of all ${\tilde C}$'s. The expected number of integer points in $U$ is then bounded by a constant times $N^4$.

On the other hand, for each given integer point in $\Lambda_N$, the probability that it is part of a Brownian loop of diameter greater than $N^\beta$ in the loop-soup is bounded by a constant times $N^{d(1-\beta)} N^{2\beta} / N^d= 1 / N^{\beta (d-2)}$. Using the independence between the set of Brownian loops of diameter smaller than $N^\beta$ and the ones with diameter greater than $N^\beta$ in the loop-soup, we therefore see that the expected number of integer points in $U$ that are also on at least one Brownian loop of diameter in $N^\beta$ is bounded by a constant times
$N^4  / N^{\beta (d-2)}$ that indeed goes to $0$ as $N \to \infty$. But if no loop of diameter greater than $N^\beta$ intersects $U$, it means that in fact, all the clusters $\tilde C$ will remain untouched and be clusters of the total loop-soup. Since they then do not contain any loop of diameter greater than $N^\beta$, this concludes the proof.
\end {proof}

\section {Concluding remarks}
We conclude with the following list of comments:
\begin {enumerate}
\item
It is easy easy to argue that the bound on $\beta$  in the previous section are sharp, i.e., that (with probability that goes to $1$ as $N \to \infty$) for any given $\beta' < 4 / (d-2)$, every cluster in ${\mathcal C}(\eps, N)$ will contain Brownian loops of diameter larger than $N^{\beta'}$. The bound on $\gamma_0$ seems to be sharp as well [i.e., for any $\gamma' < 2 / (d-4)$, loops of diameter smaller than $N^{\gamma'}$ will not be sufficient to create long cycles] but this requires more arguments and  will be the topic of another paper.
\item
In some sense, apart from having the property of containing a large cycle, the structure of the clusters in ${\mathcal C}(\eps, N)$ with no large Brownian loop can be thought of as that of ``typical'' large clusters. In a similar way as in the final Section \ref {S5}, one can see that a ``typical'' large cluster will have backbones created by loops of diameter smaller than $N^{\gamma}$ and will contain no loop of diameter greater than $N^\beta$ [which provides simple back-of-the envelope type heuristic explanations as to where the thresholds for $\beta$ and $\gamma$ come from].
\item One could keep track of bounds of the various probabilities of events in the proof in order to get a quantitative (most likely a negative power of $N$) upper bound for the probabilities such as $u_N$ in the theorem. One could also in the same spirit easily get similar results for ${\mathcal C}(N^{-w}, N)$ instead of ${\mathcal C}( \eps, N)$ provided that $w$ is chosen to be sufficiently small (so, one describes this time a polynomially large number of clusters).
\item Let us consider in $\Lambda_N$ a Poisson point process of Brownian loops, but restricted to the set of loops of diameter smaller than $N^{\gamma}$. Our result shows that for this collection of loops (that are all of size much smaller than $N$),  a tight number of clusters will contain macroscopic cycles, and that these cycles are close to be distributed like those of a Brownian loop-soup in $\Lambda_N$. This suggests that this phenomenon might be valid for any ``short-range'' critical percolation models in high dimensions. This line of thought has been developed in \cite {CW} (and forthcoming papers on the subject) that show that aspects of this indeed hold in those cases for which the two-point estimates asymptotics have been established (so, this includes Bernoulli percolation for $d \ge 11$ or some spread-out percolation models for $d \ge 7$).
\item With some additional work, one could probably use the same ideas to have the very same result for a wider class of ``cycle-containing'' clusters than the ones that avoid and wind around some affine subspace $\Delta$. Since this would not highlight any other class of cycle-containing clusters, we did not bother to do it in the present paper (the rationale being that if a cluster contains a large proper cycle for a modified definition, we would end up showing that it will contain (with positive probability) the trace of a large Brownian loop in the loop-soup, that does anyway (with high probability) contain a large winding around a $(d-2)$-dimensional affine space, so this cluster is in a way already captured by our results about the class ${\mathcal C}(\eps, N)$).
\item For each given Brownian loop $B$ in $\Lambda_N$, one can sample an independent loop-soup in $\Lambda_N$ and then look at the loop-soup cluster $c(B)$ containing $B$ of the overlay of $B$ with this loop-soup. If $B$ is defined according to the Brownian loop-measure, then this procedure defines a measure on clusters $c(B)$. Similarly, for any fixed $\eps$, if $B$ is defined according to the Brownian loop-measure restricted to ${\mathcal B}(\eps, N)$, this defines a measure $\rho_{\eps}$ on the set of clusters in $\Lambda_N$. The fact that (with probability that goes to $1$ as $N \to \infty$) no macroscopic loop-cluster contains more than one loop in ${\mathcal B}(\eps, N)$ and dimension considerations readily imply that when $d \ge 9$ (up to an event of probability that vanishes as $N \to \infty$), the collection of clusters that do contain loops in ${\mathcal B}(\eps, N)$ is distributed like a Poisson point process with intensity $\rho_{\eps}$ (and that the clusters in this Poisson point process will be disjoint). Combined with our theorem, this actually indicates that the clusters in ${\mathcal C}(\eps, N)$ are (asymptotically) distributed like a Poisson point process with intensity $2 \rho_\eps$, and that all these clusters on cable-graphs will be disjoint [mind that we are here discussing about the point process of entire clusters, not just the loops, which explains the $d \ge 9$ constraint].
\item We have stated all the result in the paper for the $\Z^d$ and its cable-graph, but the results and the proofs are obviously still valid for other $d$-dimensional lattices (as long as the Brownian motion on the cable-graph converges to continuum Brownian motion in the scaling limit).
\item
We can note that in the scaling limit (see \cite {LT}), rescaled large Brownian loops on the cable-graph $\Lambda_N$ converge to continuum Brownian loops in $[-1,1]^d$, so that the theorem implies that the rescaled collection $(\Gamma(C_1)/N, \ldots, \Gamma (C_{K_\eps})/N)$ of cycles of clusters in ${\mathcal C}(\eps, N)$ will converge to the corresponding Poissonian collection of Brownian loops (i.e., that wind around some $\Delta$ while staying at distance at least $\eps$ from it), when the intensity of the loop-soup is twice that of the usual one.
\item
As mentioned in Section \ref {LSdef}, instead of working with the loop-soup in $\Lambda_N$, one could work with the loop-soup in $\Z^d$  and consider the collection of all large cycles (i.e., of diameter at least $\eps_0 N$) that are subsets of $\Lambda_N$ and obtain similar results. Let us briefly indicate one way to proceed: The probability that there exists a Brownian loop in the loop-soup that intersects $\Lambda_N$ and goes out of $\Lambda_{mN}$ goes to $0$ as $m \to \infty$ uniformly with respect to $N$. So, for any very small $\upsilon$, we can first choose $m=m(\upsilon)$ so that this probability is smaller than $\upsilon/2$ for all large $N$, and then note that for $N$ large enough the quantity $u_{mN}$ in Theorem \ref {thm} for $\eps = \eps_0 / m$ is smaller than $\upsilon/2$.
This readily shows that large cycles (of diameter at least $\eps N$) within $\Lambda_N$ within the whole-space loop-soup clusters correspond either to a chain of small loops (of diameter smaller than $N^\beta$) or to the presence of a macroscopic Brownian loop with respective conditional probabilities that are $2\upsilon/2 = \upsilon$ close to $1/2$.
\end {enumerate}

\subsection*{Acknowledgments.} The research of WW has been supported by a Research Professorship of the Royal Society.

\end{document}